\documentclass[a4paper,9pt]{amsart}

\usepackage{amsmath,amssymb,amsfonts,latexsym}
\usepackage[T1]{fontenc} 
\usepackage[latin1]{inputenc} 
\usepackage{enumerate}
\usepackage{hyperref}
\usepackage{enumitem}   
\usepackage{color}
\usepackage[all,cmtip]{xy}
\usepackage[linesnumbered,ruled,vlined]{algorithm2e}
\usepackage{xparse}
\usepackage{stackrel}
\usepackage{parskip}
\usepackage{blkarray}
\usepackage{stmaryrd}

\usepackage{mathrsfs}

\DontPrintSemicolon

\theoremstyle{definition}
\newtheorem{teor}{Theorem}[section]
\newtheorem{prop}[teor]{Proposition}
\newtheorem{cor}[teor]{Corollary}
\newtheorem{lem}[teor]{Lemma}
\newtheorem{defi}[teor]{Definition}

\newtheorem{ej}[teor]{Example}

\newtheorem{notation}[teor]{Notation}
\newtheorem{rem}[teor]{Remark}
\newtheorem{dyn}[teor]{Definition and Notation}

\newtheorem{nyl}[teor]{Lemma and Notation}
\newtheorem{cyn}[teor]{Corollary and Notation}
\newtheorem{main}[teor]{Main Theorem}
\newtheorem{Reminder}[teor]{Reminder }

\newcommand\N{\mathbb N}

\newcommand\R{\mathbb R}
\newcommand\C{\mathbb C}

\DeclareMathOperator\id{Id}
\DeclareMathOperator\rk{rank}
\DeclareMathOperator\ev{ev}

\DeclareMathOperator\poly{poly}
\title[Truncated GNS construction]{Detecting optimality and extracting solutions in polynomial optimization with the truncated GNS construction}
\author{Mar\'ia L\'opez Quijorna}

\begin{document}

\begin{abstract}

A basic closed semialgebraic subset of $\R^{n}$ is defined by simultaneous polynomial inequalities $p_{1}\geq0,\ldots,p_{m}\geq 0$. We consider Lasserre's relaxation hierarchy to solve the problem of minimizing a polynomial over such a set. These relaxations give an increasing sequence of lower bounds of the infimum. In this paper we provide a new certificate for the optimal value of a Lasserre relaxation be the optimal value of the polynomial optimization problem. This certificate is that a modified version of an optimal solution of the Lasserre relaxation is a generalized Hankel matrix. This certificate is more general than the already known certificate of an optimal solution being flat. In case we have optimality we will extract the potencial minimizers with a truncated version of the Gelfand-Naimark-Segal construction on the optimal solution of the Lasserre relaxation. We prove also that the operators of this truncated construction commute if and only if the matrix of this modified optimal solution is a generalized Hankel matrix. This generalization of flatness will bring us to reprove a result of Curto and Fialkow on the existence of quadrature rule if the optimal solution is flat and a result of Xu and Mysovskikh on the existance of a Gaussian quadrature rule if the modified optimal solution is generalized Hankel matrix. At the end, we provide a numerical linear algebraic algorithm for dectecting optimality and extracting solutions of a polynomial optimization problem.
\end{abstract}

\maketitle

\setlength{\parindent}{0pt}

\section{Notation}

Throughout this paper, we suppose $n\in\N=\{1,2,\ldots\}$ and abbreviate $(X_{1},\ldots,X_{n})$ by $\underline{X}$.  We let $\R[\underline{X}]$ denote the ring of real polynomials in n indeterminates. We  denote $\N_{0}:=\N\cup\{0\}$. For $\alpha\in\N^{n}_{0}$, we use the standard notation :
\begin{center}
$|\alpha|:=\alpha_{1}+ \cdots + \alpha_{n} $ and $\underline{X}^{\alpha}:=X_{1}^{\alpha_{1}}\cdots X_{n}^{\alpha_{n}}$
\end{center}
For a polynomial $p\in\R[\underline{X}]$ we denote $p=\sum_{\alpha}p_{\alpha}\underline{X}^{\alpha}$ ($a_{\alpha}\in\R$). For $d\in\N_{0}$, by the notation $\R[\underline{X}]_{d}:=\{\sum_{|\alpha|\leq d}a_{\alpha}\underline{X}^{\alpha}\text{ } |\text{ } a_{\alpha}\in\R \}$  we will refer to the vector space of polynomials with degree less or equal to $d$. Polynomials all of whose monomials have exactly the same degree $d\in\N_{0}$ are called $d$-forms.  They form a finite dimensional vector space that we will denote by:
\begin{equation}
\R[\underline{X}]_{=d}:=\{\sum_{|\alpha|=d}a_{\alpha}\underline{X}^{\alpha}\text{ }|\text{ }a_{\alpha}\in\R \} \nonumber
\end{equation}
so that
\begin{equation}
\R[\underline{x}]_{d}=\R[\underline{X}]_{0}\oplus\cdots\oplus\R[\underline{X}]_{d}.\nonumber
\end{equation}
We will denote by $s_{k}:=\dim \R[\underline{X}]_{k}$ and by $r_{k}:=\dim \R[\underline{X}]_{=k}$. For $d\in\N_{0}$ we denote $\R[\underline{X}]^{*}_{d}$ the dual space of $\R[\underline{X}]_{d}$ i.e. the set of linear forms from  $\R[\underline{X}]_{d}$ to $\R$ and for $\ell\in\R[\underline{X}]_{2d}^{*}$ we denote by $\ell':=\ell_{|\R[\underline{X}]_{2d-2}}$ the restricction of the linear form $\ell$ to the space $\R[\underline{X}]_{2d-2}$. For $d\in\N_{0}$ and $a\in\R^{n}$ we denote $\ev_{a}\in\R[\underline{X}]^{*}_{d}$ the linear form such that for all $p\in\R[\underline{X}]_{d}$, $\ev_{a}(p)=p(a)$.

\section{Introduction }

Let polynomials $f,p_{1},\ldots,p_{m}\in\R[\underline{X}]$ with $m\in\N_{0}$ be given. A polynomial optimization problem involves finding the infimum of $f$ over the so called basic closed semialgebraic set $S$, defined by:
\begin{equation}\label{semialgebraic}
S:=\{x\in\R^{n} |\text{ }  p_{1}(x)\geq 0,\ldots,p_{m}(x)\geq 0 \}
\end{equation}
and also, if it is possible, a polynomial optimization problem involves extracting optimal points or minimizers i.e. elements in the set:
\begin{center}
$S^{*}:=\{x^{*}\in S$ | $\forall x\in S$ $f(x^{*})\leq f(x)\}$
\end{center}

So from now on we will denote as $(P)$, to refer us to the above defined polynomial optimization problem, that is to say:
\begin{equation}\label{Pop}
\begin{aligned}
(P) \text{  minimize  } f(x) \text{ subject to } x\in S
\end{aligned}
\end{equation}
The optimal value of $(P)$, i.e. the infimum of $f(x)$ where $x$ ranges over all feasible solutions $S$ will be denoted by $P^{*}$, that is to say:
\begin{equation}
\begin{aligned}
&P^{*}:= 
&& \inf\{\text{ }f(x)\text{ } |\text{ } x\in S  \}\in\{-\infty\}\cup\R\cup\{\infty\}
\end{aligned}
\end{equation}

Note that  $P^{*}=+\infty$ if $S=\emptyset$ and $P^{*}=-\infty$ if and only if $f$ is unbounded from below on $S$, for example if $S=\R^{n}$ and $f$ is of odd degree.\\

For $d\in\N_{0}$ let us define: 
\begin{align}
    V_{d}:= (&1,X_{1}, X_{2}, \ldots,  X_{N},X_{1}^2,X_{1}X_{2},\ldots,X_{1}X_{n},\\
		&X_{2}^2,X_{2}X_{3},\ldots,X_{n}^2,\ldots,\ldots,X_{n}^d)^{T} \nonumber
		\end{align}
as a basis for the vector space of polynomials in $n$ variables of degree at most $d$. Then
\begin{center}
$V_{d}V_{d}^{T}=\left(
\begin{array}{ccccc}
 1         & X_{1}         & X_{2}          &  \cdots &    X_{n}^{d}    \\
 X_{1}     & X_{1}^2       & X_{1}X_{2}     &  \cdots &    X_{1}X_{n}^{d}   \\
 X_{2}     & X_{1}X_{2}    & X_{2}^2        &  \cdots &    X_{2}X_{n}^{d} \\
 \vdots    & \vdots        & \vdots         &  \ddots &    \vdots \\
 X_{n}^{d} &X_{1}X_{n}^{d} & X_{2}X_{n}^{d} &  \cdots &    X_{n}^{2d}\\
    \end{array}
   \right)\in\R[\underline{X}]_{2d}^{s_{d}\times s_{d}}$
	\end{center}

Let us substitute for every monomial $\underline{X}^{\alpha}\in\R[\underline{X}]_{2d}$ a new variable $Y_{\alpha}$. This matrix has the following form: 
\begin{equation}\label{generalizedHankel}
M_{d}:=\left(
\begin{array}{ccccc}
 Y_{(0,\ldots,0)}          & Y_{(1,\ldots,0)}               & Y_{(0,1,\ldots,0)}          &  \cdots &    Y_{(0,\ldots,1)}    \\
 Y_{(1,\ldots,0)}         & Y_{(2,\ldots,0)}              & Y_{(1,1,\ldots,0)}          &  \cdots &     Y_{(1,\ldots,d)}   \\
 Y_{(0,1,\ldots,0)}       & Y_{(1,1,\ldots,0)}            &  Y_{(0,2,\ldots,0)}          &  \cdots &   Y_{(0,1,\ldots,d)}   \\
 \vdots                    & \vdots        & \vdots         &  \ddots &    \vdots \\
Y_{(0,\ldots,d)}       & Y_{(1,\ldots,d)}     & Y_{(0,1,\ldots,d)}    &  \cdots &   Y_{(0,\ldots,2d)}   \\
    \end{array}
   \right)\in\R[\underline{Y}]_{1}^{s_{d}\times s_{d}}
\end{equation}

\begin{defi} Every matrix $M\in\R^{s_{d}\times s_{d}}$ with the same shape than the matrix \eqref{generalizedHankel} is called a generalized Hankel matrix of order $d$. We denote the affine linear space of generalized Hankel matrix of order d by:
\begin{equation}\label{ghankel}
H_{d}:=\{\text{ }M_{d}(y)\text{ }|\text{ }y\in\R^{s_{2d}} \}\nonumber
\end{equation}
\end{defi}

For $p\in\R[\underline{X}]_{k}$ denote $d_{p}:=\lfloor\frac{k-\deg{p}}{2}\rfloor$  and consider the following symmetric matrix: 
\begin{equation}\begin{aligned}\label{localizing}
pV_{d_{p}}^{t}V_{d_{p}}= 
\left(\begin{array}{ccccc}
p & pX_{1} & pX_{2} & \cdots & pX_{n}^{d_{p}}\\
pX_{1} & pX_{1}^2 & pX_{1}X_{2} & \cdots & pX_{1}X_{n}^{d_{p}}\\
pX_{2} & pX_{2}X_{1} &  pX_{2}^{2} & \cdots &  pX_{2}X_{n}^{d_{p}}\\
\vdots & \vdots & \vdots & \ddots & \vdots \\
pX_{n}^{d_{p}} &  pX_{1}X_{n}^{d_{p}} &  pX_{n}^{d_{p}}X_{2} & \cdots & pX_{n}^{2d_{p}}
\end{array}\right)\in\R[\underline{X}]_{k}^{s_{d_{p}}\times s_{d_{p}}}
\end{aligned}\end{equation}
\begin{defi}
For  $p\in\R[\underline{X}]_{k}$ the localizing matrix of $p$ of degree $k$ is the matrix  resulting from substitute every monomial $\underline{X}^{\alpha}$ such that $|\alpha|\leq k$ in \eqref{localizing} for a new variable $Y_{\alpha}$. We denote this matrix by $M_{k,p}\in\R[\underline{Y}]_{1}^{s_{d_{p}}\times s_{d_{p}}}$.
\end{defi}
\begin{defi}
For a $t\times t$ real symmetric matrix $A$, the notation $A\succeq 0$ means that $A$ is positive semidefinite, i.e. $a^{T}Aa\geq 0$ for all $a\in\R^{t}$.
\end{defi}
\setlength{\parskip}{0cm}
 In order to give further characterizations of positive semidefiniteness, let us remember a very well know  theorem in linear algebra.
\begin{Reminder}\label{spectral}
Suppose $A\in\R^{t\times t}$ is symmetric. Then there is a diagonal matrix $D\in\R^{t\times t}$ and $U\in\R^{t\times t}$ orthogonal matrix, i.e. $UU^{T}=U^{T}U=I_{t}$, such that $U^{T}AU=D$ 
\end{Reminder}
\begin{Reminder}\label{mpsd}
Let $A\in\R^{t\times t}$ symmetric. The following are equivalent:
\begin{enumerate}
\item $A\succeq 0$.
\item All eigenvalues of $A$ are nonnegative.
\item There exists $B\in\R^{t\times t}$ such that $A=B^{T}B$.
\end{enumerate}
\end{Reminder}

\begin{proof}
$(1)\implies(2)$. Suppose $A\succeq 0$ and take $\lambda\in\R$ and eigenvalue of $A$ such that $Av=\lambda v$ for $v\in\R^{t}$ eigenvector of $A$, then $v\neq 0$. impplying that $0\leq v^{T}Av=v^{T}\lambda v=\lambda\lVert v \rVert^{2} $ then $\lambda\geq 0$.\\
$(2)\implies(3)$. Suppose that all eigenvalues of $A$ are nononnegative then by \ref{spectral} there exits $D$ diagonal matrix with nonnegative entries $\lambda_{i}\geq 0$ and $U$ orthogonal matrix such that $U^{T}AU=D$. Take $B:=RU^{T}\in\R^{t\times t}$ where $R$ is the diagonal matrix which entries are $\sqrt{\lambda_{i}}$.\\
$(3)\implies(1)$. Suppose there is $B\in\R^{t\times t}$ such that $A=B^{T}B$, take $a\in\R^{t}$ then  $a^{t}Aa=a^{t}B^{T}Ba=(Ba)(Ba)=\lVert Ba \rVert^{2}\geq 0$.
\end{proof}

\begin{defi}\label{relax}
Let $(P)$ be a polynomial optimization problem  as in \eqref{Pop} and 
let $k\in$ $\N_{0}\cup\{\infty\}$ such that $f,p_{1},\ldots, p_{m}\in\R[\underline{X}]_{k}$. The Moment relaxation (or Lasserre relaxation) of $(P)$ of degree $k$ is the following semidefinite optimization problem:
\begin{align}
&(P_{k})
\text{ minimize }
\sum_{|\alpha|\leq k} f_{\alpha}y_{\alpha} \text{ subject to } \\ 
&
M_{k,1}(y)\succeq 0, \ \ y_{\left(0,\ldots,0\right)}=1, \ \ M_{k,p_{i}}(y)\succeq 0 \nonumber
\end{align}
the optimal value of $(P_{k})$ that is to say, the infimum over all
\begin{equation}
 y=(y_{\left(0,\ldots,0 \right)},\ldots,y_{\left(0,\ldots,k \right)})\in\R^{s_{k}}\nonumber
 \end{equation} that ranges over all feasible solutions of $(P_{k})$ is denoted by $P^{*}_{k}\in\{-\infty\}\cup\R\cup\{\infty\} $.
\end{defi}

Given a polynomial optimization problem $(P)$ as in \eqref{Pop} and $M:=M_{d}(y)\in\R^{s_{d}\times s_{d}}$  an optimal solution of $(P_{2d})$, it is always  possible to find a matrix $W_{M}\in\R^{s_d\times r_{d}}$ such that $M$ can be decomposed in a block matrix of the following form  (see \ref{usefull} below for a proof):
$$M=
\left(
\begin{array}{c|c}
\makebox{$A_{M}$}&\makebox{$A_{M}W_{M}$}\\
\hline
  \vphantom{\usebox{0}}\makebox{$W^{T}_{M}A_{M}$}&\makebox{$C_{M}$}
\end{array}
\right)
$$
This useful result can be also found in \cite{Smu} and in \cite[Lemma 2.3]{recursiv}.
Define the following matrix:
\begin{equation}
\widetilde{M}:=\left(
\begin{array}{c|c}
\makebox{$A_{M}$}&\makebox{$A_{M}W_{M}$}\\
\hline
  \vphantom{\usebox{0}}\makebox{$W^{T}_{M}A_{M}$}&\makebox{$W^{T}_{M}A_{M}W_{M}$}
\end{array}\right)\nonumber
\end{equation}
In this paper we prove that $\widetilde{M}$ is well-defined, that is to say it does not depend from the election of $W_{M}$, and assuming that $W^{T}_{M}A_{M}W_{M}$ is a generalized Hankel matrix we will use a new method to find a decomposition:
\begin{equation}\label{factorization1}
\widetilde{M}=\sum_{i=1}^{r}\lambda_{i}V_{d}(a_{i})V_{d}(a_{i})^{T}
\end{equation}

where $r:=\rk M$ ,$a_{1},\ldots,a_{r}\in\R^{n}$ and $\lambda_{1}>0,\ldots,\lambda_{r}>0$. In this paper we will show that for some polynomial optimization problems if we have that $W_{M}^{T}A_{M}W_{M}$ is generalized Hankel and the nodes are contained in $S$, even if $M$ is not flat i.e. $W_{M}^{T}A_{M}W_{M}\neq C_{M}$ (see the definition in \ref{definiciones}), we can still claim optimality, that is to say that $a_{1},\ldots,a_{r}$ are global minimizers. We will also see some examples to discard optimality or in other words to discard that $M$ has a factorization as in \eqref{factorization1}, see \ref{madrugada}. Let us advance two results concerning optimality.

\begin{teor}\label{optimality}
Let $(P)$ be a polynomial optimization problem as in \eqref{Pop} and suppose that $M_{d}(y)\in\R^{s_{d}\times s_{d}}$ is an optimal solution of $(P_{2d})$ and $\widetilde{M_{d}(y)}$ is a generalized Hankel matrix. Then there are $a_{1},\ldots,a_{r}\in\R^{n}$ points and $\lambda_{1}>0,\ldots,\lambda_{r}>0$ weights such that:
\begin{equation}\label{auxiliar}
\widetilde{M_{d}(y)}=\sum_{i=1}^{r}\lambda_{i}V_{d}(a_{i})V_{d}(a_{i})^{T}
\end{equation}
where $r=\rk A_{M}$. Moreover if $\{a_{1},\ldots,a_{r}\}\subseteq S$ and $f\in\R[\underline{X}]_{2d-1}$ then $a_{1},\ldots,a_{r}$ are global minimizers of $(P)$ and $P^{*}=P^{*}_{2d}=f(a_{i})$ for all $i\in\{1,\ldots,r \}$.
\end{teor}
\begin{proof}
The correspondence given in \ref{traduccion} together with the Theorem \ref{popcorn} will give us the proof.
\end{proof}
\begin{rem}\label{nonconstrained}
Let $(P)$ be a polynomial optimization problem without constraints. Suppose $M_{d}(y)\in\R^{s_{d}\times s_{d}}$ is an optimal solution of $(P_{2d})$ with $\widetilde{M_{d}(y)}$ a generalized Hankel matrix and that $f\in\R[\underline{X}]_{2d-1}$. Applying Theorem \ref{optimality} we get the decomposition \eqref{auxiliar}, and since we can ensure that $a_{1},\ldots,a_{r}\subseteq S =\R^{n}$ then they are global minimizers of $(P)$ and $P^{*}=P^{*}_{2d}=f(a_{i})$ for all $i\in\{1,\ldots,r \}$.
\end{rem}

\begin{ej}\label{porfavor}
Let us considerer the following polynomial optimization problem taken from \cite[Problem 4.7]{griegos}:
\begin{equation*}
\begin{aligned}
& {\text{minimize}}
& & f(\underline{x})=-12x_{1}-7x_{2}+x_{2}^2 \\
& \text{subject to}
& & -2x_{1}^4+2-x_{2}=0\\
& 
& & 0\leq x_{1}\leq 2\\
&
&& 0\leq x_{2}\leq 3\\
\end{aligned}
\end{equation*}
We get the optimal value $P^{*}_{4}=-16.7389$ associated to the following optimal solution:

\begin{equation}\label{segundoej}
M_{2}(y) =  \begin{blockarray}{ccccccc}
\text{ } & 1 & X_{1} & X_{2} & X_{1}^2 & X_{1}X_{2} & X_{2}^2 \\
\begin{block}{r(rrr|rrr)}
 1 &   1.0000  &  0.7175  &  1.4698 &   0.5149  &  1.0547  &  2.1604  \\
 X_{1} &   0.7175  &  0.5149  &  1.0547 &   0.3694  &  0.7568  &  1.5502   \\
 X_{2} &  1.4698  &  1.0547  &  2.1604 &   0.7568  &  1.5502  &  3.1755  \\ \cline{2-7}
 X_{1}^2 &   0.5149  &  0.3694  &  0.7568 &   0.2651  &  0.5430  &  1.1123  \\
 X_{1}X_{2} &   1.0547  &   0.7568 &  1.5502 &   0.5430  &  1.1123 &   2.2785   \\
 X_{2}^2 &   2.1604  &  1.5502  &  3.1755 &   1.1123  &  2.2785  &  8.7737   \\
   \end{block}
   \end{blockarray}
\end{equation}

and the modified moment matrix of $M_{2}(y)$ is the following:
\begin{equation}
\widetilde{M_{2}(y)} = \left(
\begin{array}{rrr|rrr}
    1.0000  &  0.7175  &  1.4698 &   0.5149  &  1.0547  &  2.1604 \\
    0.7175  &  0.5149  &  1.0547 &   0.3694  &  0.7568  &  1.5502  \\
    1.4698  &  1.0547  &  2.1604 &   0.7568  &  1.5502  &  3.1755 \\\cline{1-6}
    0.5149  &  0.3694  &  0.7568 &   0.2651  &  0.5430  &  1.1123 \\
    1.0547  &   0.7568 &   1.5502 &   0.5430  &  1.1123 &   2.2785 \\
    2.1604  &  1.5502  &  3.1755 &   1.1123  &  2.2785  &  \textbf{4.6675} 
   \end{array}
    \right)
\end{equation}

We get that $\widetilde{M_{2}(y)}$ is a generalized Hankel matrix and $f\in\R[X_{1},X_{2}]_{3}$ to conclude optimality, according with Theorem \ref{optimality}, it remains to calculate the factorization \eqref{factorization1} and check if the points are in $S$.  We will see in Section $5$ in \ref{volver} how to compute this factorization, in this case, it is easy to see that:
\begin{equation}
\widetilde{M_{2}(y)}=V_{2}(\alpha,\beta)V_{2}(\alpha,\beta)^{T}\nonumber
\end{equation}
where $\alpha:=0.7175$ and $\beta:=1.4698$. One can verify that $(\alpha,\beta)\in S$ and therefore we can conclude that $P^{*}_{4}=P^{*}=-16.7389$ is the optimal value and $(\alpha,\beta)$ is a minimizer.

\end{ej}

\begin{teor}\label{mpolyeder}
Let $(P)$ be a polynomial optimization problem given as in \eqref{Pop} and suppose that the $p_{i}$ from \eqref{semialgebraic} are all of degree at most $1$ (so that $S$ is a polyhedron). Suppose that $M_{d}(y)\in\R^{s_{d}\times s_{d}}$ is an optimal solution of $(P_{2d})$ and that $\widetilde{M_{d}(y)}$ is a generalized Hankel matrix. Then there are $a_{1},\ldots,a_{r}\in S$ and $\lambda_{1}>0,\ldots,\lambda_{r}>0$ weights such that:
\begin{equation}
\widetilde{M_{d}(y)}=\sum_{i=1}^{r}\lambda_{i}v_{d}(a_{i})v_{d}(a_{i})^{T} \nonumber
\end{equation}
Moreover if $f\in\R[\underline{X}]_{2d-1}$ then $a_{1},\ldots,a_{r}$ are global minimizers of $(P)$ and $P^{*}=P^{*}_{2d}=f(a_{i})$ for all $i=1,\ldots,r$.
\end{teor}
\begin{proof}
The correspondence given in Corollary \ref{traduccion} together with the Theorem \ref{poliedro} will give us the result.
\end{proof}

\begin{ej}\label{amigo}
Let us consider the following polynomial optimization problem, taken from \cite[page 18]{tonto}, whose objective function is the Moztkin polynomial \cite[Prop.1.2.2]{mar}:
\begin{equation*}
\begin{aligned}
& {\text{minimize}}
& & f(x)=x_{1}^{4}x_{2}^{2}+x_{1}^2x_{2}^{4}-3x_{1}^2x_{2}^2+1 \\
& \text{subject to}
& & -2\leq x_{1}\leq 2\\
& 
& & -2\leq x_{2}\leq 2\\  
\end{aligned}
\end{equation*}
We get the optimal value  $P^{*}_{8}=6.2244\cdot 10^{-9}$ from the following optimal solution of $(P_{8})$:

\begin{equation}\label{noseque}
M:=M_{8,1}(y)=\left(
\begin{array}{c|c}
\makebox{$A_{M}$}&\makebox{$A_{M}W_{M}$}\\
\hline
  \vphantom{\usebox{0}}\makebox{$W^{T}_{M}A_{M}$}&\makebox{$C_{M}$}
\end{array}\right)
\end{equation}

where:

{\tiny
\begin{equation}
 A_{M} =\begin{blockarray}{ccccccccccc}
 \text{ } & 1 & X_{1} & X_{2} & X_{1}^{2} & X_{1}X_{2} & X_{2}^{2} & X_{1}^3 & X_{1}^{2}X_{2} & X_{1}X_{2}^{2} & X_ {1}^{3}  \\
\begin{block}{r(rrrrrrrrrr)}
1 &     1.0000  & -0.0005 &  -0.0004 &   1.0000  & -0.0000  &  1.0000 &  -0.0005  & -0.0004  & -0.0005 &  -0.0004 \\
X_{1} &   -0.0005  &  1.0000 &  -0.0000 &  -0.0005  & -0.0004  & -0.0005 &   1.0000  & -0.0000  &  1.0000 &  -0.0000 \\
X_{2} &   -0.0004  & -0.0000 &   1.0000 &  -0.0004  & -0.0005  & -0.0004 &  -0.0000  &  1.0000  & -0.0000 &   1.0000 \\
 X_{1}^{2} &   1.0000  &  -0.0005 &  -0.0004 &   1.0000 &  -0.0000 &   1.0000 &  -0.0005 &  -0.0004 &  -0.0005  & -0.0004 \\
 X_{1}X_{2} &  -0.0000  & -0.0004  & -0.0005 &  -0.0000 &   1.0000 &  -0.0000 &  -0.0004 &  -0.0005 &  -0.0004  & -0.0005 \\
 X_{2}^{2} &   1.0000  & -0.0005 &   -0.0004  &  1.0000  & -0.0000 &   1.0000  & -0.0005 &  -0.0004  & -0.0005 &  -0.0004 \\
X_{1}^{3} &   -0.0005  &  1.0000 &  -0.0000 &  -0.0005 &  -0.0004 &  -0.0005 &   1.0001 &  -0.0000 &   1.0001 &  -0.0000 \\
X_{1}^{2}X_{2} &  -0.0004  & -0.0000 &   1.0000 &  -0.0004 &  -0.0005 &  -0.0004 &  -0.0000  &  1.0001  & -0.0000 &   1.0001 \\
 X_{1}X_{2}^{2} &  -0.0005  &  1.0000 &  -0.0000 &  -0.0005 &  -0.0004 &  -0.0005 &   1.0001 &  -0.0000 &   1.0001 &  -0.0000 \\
 X_{2}^{3} &  -0.0004  &  -0.0000 &   1.0000 &  -0.0004 &  -0.0005 &  -0.0004 &  -0.0000 &   1.0001  & -0.0000 &   1.0001 \\
\end{block}
\end{blockarray}
			\end{equation}}
			
{\scriptsize	
\begin{equation}
 W_{M} = \left(\begin{array}{ccccc}
 1 & 0 & 1 & 0 & 1\\
 0 & 0 & 0 & 0 & 0\\
 0 & 0 & 0 & 0 & 0\\
 0 & 0 & 0 & 0 & 0\\
0 & 1 & 0 & 1 & 0\\
 0 & 0 & 0 & 0 & 0\\
 0 & 0 & 0 & 0 & 0\\
 0 & 0 & 0 & 0 & 0\\
 0 & 0 & 0 & 0 & 0\\
 0 & 0 & 0 & 0 & 0
\end{array}\right)
    \text{ and }
	C_{M}	=	\begin{blockarray}{cccccc}
	\text{ } & X_{1}^{4} & X_{1}^{3}X_{2} & X_{1}^{2}X_{2}^{2} & X_{1}X_{2}^{3} & X_{1}^{4} \\
\begin{block}{r(rrrrr)}
 X_{1}^{4} &   6.4115 &  -0.0000  &  2.0768 &  -0.0000 &   1.7719 \\
  X_{1}^{3}X_{2} & -0.0000  &  2.0768  & -0.0000 &   1.7719 &  -0.0000 \\
 X_{1}^{2}X_{2}^{2} &   2.0768  & -0.0000  &  1.7719 &  -0.0000  &  2.0768 \\
X_{1}X_{2}^{3} &   -0.0000  &  1.7719  &  -0.0000  &  2.0768 &  -0.0000 \\
 X_{2}^{4} &   1.7719  & -0.0000  &  2.0768 &  -0.0000 &   6.4115 \\
\end{block}
\end{blockarray}    	
	\end{equation}}
	
	In this case:
	{\scriptsize
\begin{align}
W^{*}A_{M}W = \left( \begin{array}{rrrrr}
    1.0000  & -0.0000 &   1.0000 &  -0.0000  &  1.0000 \\
   -0.0000  &  1.0000 &  -0.0000 &   1.0000  & -0.0000 \\
    1.0000  & -0.0000 &   1.0000 &  -0.0000  &  1.0000 \\
   -0.0000  &  1.0000 &  -0.0000 &   1.0000  & -0.0000 \\
    1.0000  &  -0.0000 &   1.0000 &  -0.0000 &   1.0000 
		\end{array}
		\right) \nonumber
\end{align}	}

 is a Hankel matrix, what  implies that $\widetilde{M}$ is generalized Hankel and since we are minimizing over a polyhedron defined by linear polynomials  by Theorem \ref{mpolyeder}  $P^{*}_{8}=P^{*}$.

\end{ej}

The goal of this paper is to find optimality conditions and extracting global minimizers from an optimal solution of the moment relaxation. That is to say given a polynomial optimization problem $(P)$ as in \eqref{Pop} and an optimal solution of the  moment relaxation $(P_{k})$ as in \ref{relax}, find conditions to conclude if the optimal value is also the optimal value of the original polynomial optimization problem, i.e. $P^{*}=P^{*}_{k}$ and in this case extracting global minimizers. In the first section we outline Lasserres approach \cite{las} to solve polynomial optimization problems with the language of linear forms, at the end of this section we will reformulate the problem of optimality, that is to say we reformulate the problem of finding a decomposition of the modified moment matrix as in \eqref{auxiliar} to the problem of finding a commutative truncated version of the Gelfand-Naimark-Segal construction for a linear form $L\in\R[\underline{X}]^{*}_{2d}$ with $d\in\N_{0}\cup\{\infty\}$, which take nonnegative values in $\sum\R[\underline{X}]_{d}^{2}$. The truncated GNS construction for this linear form will be defined in Section $4$ and at the end of this section we give a proof of the very useful result of Smul'jan \cite{Smu} using the inner product defined in the truncated GNS construction. In Section $5$ we prove that if the truncated GNS multiplication operators of the optimal solution  commute we are able to get the factorization \eqref{factorization1} or in other words we find a Gaussian quadrature rule \ref{gaussian} representation for the linear form. In this section we will also prove that the commutativity of the truncated GNS operators is a more general fact than the very well know flatness condition, that is the case $C_{M}=W_{M}^{T}A_{M}W_{M}$, but the reverse it does not always hold (see \eqref{segundoej},\ref{elprimero},\ref{rosenbrock},\eqref{cuartoej}, \eqref{nopuedomas} for examples), at the end of this section we review a result of Curto and Fialkow for the characterization of linear forms with quadrature rule on the whole space with minimal number of nodes. In Section $6$ we prove the main result, which is that the truncated GNS multiplication operators of $M$ commute if and only if $W_{M}^{T}A_{M}W_{M}$ is a Hankel matrix. This fact will help us to detect optimality in polynomial optimization problems and to slightly generalize some classical results of Dunkl, Xu ,Mysovskikh, Möller and Putinar \cite[Theorem 3.8.7]{xu},\cite{mysov,moller},\cite[pages 189-190]{Put} on Gaussian quadratue rules. with underlying ideas of \cite{Put}. In the last section we group all the results about optimality and global minimizers for an optimal solution of the moment relaxation, at the end we also give an algorithm for detecting the optimality and extracting minimizers with numerical examples.

\section{Formulation of the problem}
To solve polynomial optimization problems we use the very well known moment relaxations defined in \ref{relax}. An introduction in to moment relaxations can also be found for instance in: \cite{lau},\cite{las} and \cite{sch}. Likewise we will give the equivalent definition using linear forms instead of matrices in \ref{relax2}. We will now outline Lasserre's \cite{las} approach to solve this problem. This method constructs a hierarchy of semidefinite programming relaxations, which are generalization of linear programs, and possible to solve efficiently, see \cite{ig} and \cite{lau} for an introduction. In each relaxation of degree $k$  we build  convex set, obtained through the linearization of a equivalent polynomial optimization problem of $(P)$ defined in \eqref{Pop}. This equivalent formulation of the problem consists in adding infinitely many redundant inequalites of the form $p\geq 0$ for all  $p\in \sum\R[\underline{X}]^2p_{i}\cap\R[\underline{X}]_{k}$ (with the notation $\sum\R[\underline{X}]^2p_{i}$ we mean the set of all finite sums of elements of the form $p^{2}p_{i}$, for $p\in\R[\underline{X}]$). The set of this redundant inequalities builds a cone, which is a set containing 0, closed under addition and closed under multiplication for positive scalars. The cone generated for this redundant inequalities is called truncated quadratic module generated by the polynomials $p_{1},\ldots,p_{m}$, as we see in Definition \ref{dqm}. This relaxations give us an increasing sequence of lower bounds of the infimum $P^{*}$, as you can see in \ref{motiv}. Lasserre proved that this sequence converge asymptotically to the infimum if we assume some arquimedean property in the cone generetated for the redundant inequalities, see \cite[Theorem 5]{sch} for a proof.  

\begin{defi}\label{dqm}
Let $p_{1},\ldots,p_{m}\in\R[\underline{X}]$ and $k\in \N_{0}\cup\{\infty\}$. We define the $k$-truncated quadratic module $M$, generated by $p_{1},\ldots,p_{m}$ as:
\begin{equation}\label{qm}
\begin{aligned}
M_{k}(p_{1},\ldots,p_{m}):= & \R[\underline{X}]_{k}\cap\sum\R[\underline{X}]^{2}+ \R[\underline{X}]_{k}\cap\sum\R[\underline{X}]^{2}p_{1}\\
                           & +\cdots +\R[\underline{X}]_{k}\cap\sum\R[\underline{X}]^{2}p_{m} \subseteq\R[\underline{X}]_{k} \\
\end{aligned}
\end{equation}
where here $\R[\underline{X}]_{\infty}:=\R[\underline{X}]$. We use the notation $M(p_{1},\ldots,p_{m}):=M_{\infty}(p_{1},\ldots,p_{m})$, to refer to the quadratic module generated by the polynomials $p_{1},\ldots,p_{m}\in\R[\underline{X}]$.
\end{defi}

\begin{rem}\label{gram}
Note that:
\begin{equation}
\R[\underline{X}]_{k}\cap\sum\R[\underline{X}]^{2}p=\{\sum_{i=1}^{l}h_{i}^{2}p\text{ }|\text{ }h_{i}\in\R[\underline{X}],2\deg(h_{i})\leq k -\deg(p) \}\nonumber
\end{equation}
For a proof this see \cite[Page 5]{sch}.
\end{rem}

\begin{lem}\label{lecture}
Let $k\in\N$, $p\in\R[\underline{X}]_{k}\setminus\{0\}$ and $d:=\lfloor \frac{k-\deg(p)}{2} \rfloor$.
Let $L\in\R[\underline{X}]^{*}_{k}$. Then it holds:
\begin{equation}
L(\sum\R[\underline{X}]_{k}\cap\R[\underline{X}]^{2}p)\subseteq\R_{\geq 0}\iff M_{k,p}(y)\succeq 0
\end{equation}
\end{lem}
\begin{proof}
Let us set the matrices $A_{\alpha}\in\R^{s_{d}\times s_{d}}$ for $|\alpha|\leq k$, as the matrices such that:
\begin{equation}
 pV_dV_d^{T}=\sum_{|\alpha|\leq k}\underline{X}^{\alpha}A_{\alpha}\in\R[\underline{X}]^{s_{d}\times s_{d}}_{k}.\nonumber
\end{equation}
and  $y_{\alpha}:=L(\underline{X}^{\alpha})$ for $|\alpha|\leq k$.
\begin{align}
L(\sum\R[\underline{X}]_{k}&\cap\R[\underline{X}]^{2}p)\subseteq\R_{\geq 0}\stackrel {\ref{gram}}{\iff} \forall h\in\R[\underline{X}]_{d},\text{ }L(h^{2}p)\geq 0 \nonumber \\ \nonumber
&\iff \forall H\in\R^{s_{d}},\text{ } L((H^{T}V_{d})(V_{d}^{T}H)p)\geq 0 \nonumber \\ \nonumber
&\iff \forall H\in\R^{s_{d}},\text{ } L(H^{T}pV_{d}V_{d}^{T}H)\geq 0 \nonumber \\ \nonumber
&\iff \forall H\in\R^{s_{d}},\text{ } L(H^{T}(\sum_{|\alpha|\leq k}\underline{X}^{\alpha}A_{\alpha})H)\geq 0 \nonumber \\ \nonumber
&\iff \forall H\in\R^{s_{d}},\text{ } L(\sum_{|\alpha|\leq k}\underline{X}^{\alpha}H^{T}A_{\alpha}H)\geq 0 \nonumber \\ \nonumber
&\stackrel{\text{L is linear}}{\iff} \forall H\in\R^{s_{d}},\text{ } \sum_{|\alpha|\leq k}L(\underline{X}^{\alpha})H^{T}A_{\alpha}H\geq 0 \nonumber \\ \nonumber
&\iff \forall H\in\R^{s_{d}},\text{ } H^{T}(\sum_{\alpha}y_{\alpha}A_{\alpha})H\geq 0 \nonumber \\ \nonumber
&\iff \sum_{|\alpha|\leq k}y_{\alpha}A_{\alpha}\succeq 0 \iff M_{k,p}(y)\succeq 0
\end{align}
\end{proof}

Due to Lemma \ref{lecture} the following definition of moment relaxation using linear forms is equivalent to the definition given in \ref{relax}

\begin{defi}\label{relax2}
Let $(P)$ be a polynomial optimization problem given as in \ref{Pop} and 
let $k\in$ $\N_{0}\cup\{\infty\}$ such that $f,p_{1},\ldots, p_{m}\in\R[\underline{X}]_{k}$. The moment relaxation (or Lasserre relaxation) of $(P)$ of degree $k$ is the semidefinite optimization problem:
\begin{equation*}
\begin{aligned}
&(P_{k})
&&\text{ minimize }
&&& L(f)
&&&&\text{ subject to }
&&&&& L\in\R[\underline{X}]_{k}^{*}
&&&&&&
&&&&&&&\\
&
&&
&&&
&&&&
&&&&& L(1)=1
&&&&&& 
&&&&&&&\\
&
&&
&&&
&&&&
&&&&&  L(M_{k}(p_{1},..., p_{m}))\subseteq{\R_{\geq 0}}
&&&&&&
&&&&&&&\\
\end{aligned}
\end{equation*}
the optimal value of $(P_{k})$ i.e., the infimum over all $L(f)$ where $L$ ranges over all optimal solutions of $(P_{k})$ is denoted by $P^{*}_{k}\in\{-\infty\}\cup\R\cup\{\infty\} $.
\end{defi}

\begin{cyn}\label{traduccion}
Let $d\in\N_{0}$. The correspondence:
\begin{align}
L  \nonumber\mapsto & (L(\underline{X}^{\alpha+\beta}))_{|\alpha|,|\beta|\leq d}\\ \nonumber
\left(
\begin{array}{ccc}
\R[\underline{X}]_{2d} & \rightarrow & \R \\
\underline{X}^{\alpha} & \mapsto & y_{\alpha} \\
\end{array}\right) \mapsfrom   &  M_{d}(y)\\ \nonumber
\end{align}
defines a bijection between the linear forms $L\in\R[\underline{X}]^{*}_{2d}$ such that $L(\sum\R[\underline{X}]^{2}_{d}])\subseteq \R_{\geq 0}$ and the set of positive semidefinite generalized Hankel matrices of order $d$ i.e. $ H_{d}\cap\R_{\succeq 0}^{s_{d}\times s_{d}}$. Let $L\in\R[\underline{X}]^{*}_{2d}$ such that $L(\sum\R[\underline{X}]^{2}_{d}])\subseteq \R_{\geq 0}$ we denote $M_{L}:=(L(\underline{X}^{\alpha+\beta}))_{|\alpha|,|\beta|\leq d}$ and let $M_{d}(y)\succeq 0$ for $y\in\R^{s_{d}}$ we denote:
\begin{equation}
L_{M_{d}(y)}:\R[\underline{X}]_{2d}\longrightarrow\R,\underline{X}^{\alpha}\mapsto y_{\alpha}.\nonumber
\end{equation}
\end{cyn}
\begin{proof}
The well-definedness of both maps follows from Lemma \ref{lecture}. Now, let  $L \in\R[\underline{X}]^{*}_{2d}$ such that $L(\sum\R[\underline{X}]^{2}_{d})\subseteq \R_{\geq 0}$ then:
\begin{equation}
L_{M_{L}}:\R[\underline{X}]_{2d}\longrightarrow \R,\underline{X}^{\alpha}\mapsto L(\underline{X}^{\alpha})\nonumber
\end{equation}
since notice that $M_{L}=(L(\underline{X}^{\alpha+\beta}))_{|\alpha|,|\beta|\leq d}=M_{d}(L(0),\ldots,L(X_{n}^{2d}))$. Hence $L_{M_{L}}=L$. On the other side, let $M_{d}(y)\succeq 0$ for $y\in\R^{s_{d}}$ then:
\begin{equation}
M_{L_{M_{d}(y)}}=(L_{M_{d}(y)}(\underline{X}^{\alpha+\beta}))_{|\alpha|,|\beta|\leq d}=(y_{\alpha+\beta})_{|\alpha|,|\beta|\leq d}=M_{d}(y).\nonumber
\end{equation}
\end{proof}

\begin{notation}
We denote the following isomorphism of vector spaces by:
\begin{align}
\poly:\R^{s_{d}}\longrightarrow\R[\underline{X}]_{d}\nonumber,a\mapsto a^{T}V_{d}
\end{align}
\end{notation}

\begin{prop}\label{multiplicacion}
Let $d\in\N_{0}$ and $L\in\R[\underline{X}]_{2d}^{*}$ then:
\begin{equation}
L(pq)=P^{T}M_{L}Q \nonumber
\end{equation}
where $P:=\poly^{-1}(p)$ and $Q:=\poly^{-1}(q)$.
\end{prop}
\begin{proof}
As usual let us set the matrices  $A_{\alpha}\in\R^{s_{d}\times s_{d}}$ for $|\alpha|\leq 2d$ as the matrices such that:
\begin{equation}
 V_dV_d^{T}=\sum_{|\alpha|\leq k}\underline{X}^{\alpha}A_{\alpha}\in\R[\underline{X}]^{s_{d}\times s_{d}}_{k}.\nonumber
\end{equation}
Then:
\begin{align}
L(pq)=L(P^{T}V_{d}V_{d}^{T}Q)\nonumber=L(P^{T}\sum_{|\alpha|\leq 2d}\underline{X}^{\alpha}&A_{\alpha}Q)=  \\ \nonumber
\sum_{|\alpha|\leq 2d}L(\underline{X}^{\alpha})P^{T}A_{\alpha}Q= P^{T}(\sum_{|\alpha|\leq 2d}L(\underline{X}^{\alpha})A_{\alpha})Q &=P^{T}M_{L}Q \nonumber
\end{align}
\end{proof}

\begin{defi}
Let $L\in\R[\underline{X}]^{*}_{d}$ . A quadrature rule for $L$ on $U\subseteq\R[\underline{X}]_{d}$ is a function $w:N\to\R_{>0}$ defined on a finite set $N\subseteq\R^{n}$, such that:
\begin{equation}
L(p)= \sum\limits_{x\in N} w(x)p(x)
\end{equation}
for all $p\in U$. A quadrature rule for $L$ is a quadrature for $L$ on $\R[\underline{X}]_{d}$. We call the elements of $N$ the nodes of the quadrature rule. 
\end{defi}

\begin{prop}\label{motiv}
Let $(P)$ be the polynomial optimization problem given in \eqref{Pop} with $f,p_{1},\ldots,p_{m}\in\R[\underline{X}]_{k}$. Then the following holds:
\begin{enumerate}[label=(\roman*)]
\item $P^{*}\geq P_{\infty}^{*}\geq\cdots\geq P^{*}_{k+1} \geq P^{*}_{k}$.
\item Let $L\in\R[\underline{X}]^{*}_{k}$ with $L(1)=1$. Suppose $L$ has a quadrature rule with nodes in $S$, then $L$ is a feasible solution of $(P_{k})$ with $L(f)\geq P^{*}$.
\item Suppose $(P_{k})$ has an optimal solution $L^{*}$, which has a quadrature rule on $\R[\underline{X}]_{l}$ for some $l\in\{1,\ldots,k\}$ with $f\in\R[\underline{X}]_{l}$ and the nodes are in $S$. Then $L^{*}(f)=P^{*}$, moreover we have $P^{*}=P^{*}_{k+m}$ for $m\geq 0$ and the nodes of the quadrature rule are global minimizers of $(P)$.
\item In the situation of (iii), suppose moreover that $(P)$ has an unique global minimizer $x^{*}$, then $L^{*}(f)=f(x^{*})$ and  $x^{*}=(L^{*}(X_{1}),\ldots,L^{*}(X_{n}))$.
\end{enumerate}
\end{prop}

\begin{proof}

(i) $P^{*}\geq P^{*}_{\infty}$ since if $x$ is a feasible solution for $(P)$ then $\ev_{x}\in\R[\underline{X}]^{*}$ is a feasible solution for $P_{\infty}$ with the same value, that is $f(x)=\ev_{x}(f)$. 
It remains to prove $P^{*}_{l}\geq P^{*}_{k}$ for $l\in\N_{\geq k}\cup\{ \infty \}$. For this let $L$ be  a feasible solution of $(P_{l})$, as $M_{k}(p_{1},\ldots,p_{m})\subseteq M_{l}(p_{1},\ldots,p_{m})$ 
then $L_{|\R[\underline{X}]_{k}}$ is a feasible solution of $(P_{k})$ with the same optimal value.\\
(ii) Suppose $L$ has a quadrature rule with nodes $a_{1},\ldots,a_{N}\in S$ and weights $\lambda_{1}>0,\ldots,\lambda_{N}>0$. From $L(1)=1$ we get $\sum_{i=1}^{N}\lambda_{i}=1$ and since the nodes are in $S$ it holds $L(M_{k}(p_{1},\ldots,p_{m}))\subseteq \R_{\geq 0}$. Hence $L$ is a feasible solution of $(P_{k})$. Moreover the following holds:
\begin{equation}
P^{*}=L(1)P^{*}=\sum_{i=1}^{N}\lambda_{i}P^{*}\leq\sum_{i=1}^{N}\lambda_{i}f(a_{i})=L(f)\nonumber
\end{equation}
where the inequality follows from the fact that $P^{*}\leq f(x)$ for all $x\in S$.\\
(iii) Suppose $L^{*}$ is an optimal solution of $(P_{k})$ then $L^{*}(f)=P^{*}_{k}\leq P^{*}$ using (i) and on other side since $L^{*}(1)=1$ and $L^{*}$ has a quadrature rule on $\R[\underline{X}]_{l}$ with nodes in $S$ and $f\in\R[\underline{X}]_{l}$, there exist $a_{1},\ldots,a_{N}\in S$ nodes, and $\lambda_{1}>0,\ldots,\lambda_{N}>0$ weights, such that:
\begin{equation}\label{comun}
P^{*}_{k}=L^{*}(f)=\sum_{i=1}^{N}\lambda_{i}f(a_{i})\geq\sum_{i=1}^{N}\lambda_{i}P^{*}=P^{*}
\end{equation}

Therefore $L^{*}(f)=P^{*}$, and  since $P^{*}_{k}=P^{*}$ we get equality everywhere in (i) and we can conclude  that $P^{*}=P^{*}_{k+m}$ for $m\geq 0$. It remains to show that the nodes are global minimimizers of $(P)$, but this is true since in \eqref{comun} we have equality everywhere, and if we factor out we get $\sum_{i=1}^{N}\lambda_{i}(f(a_{i})-P^{*})=0$, as $\lambda_{i}>0$ and $f(a_{i})-P^{*}\geq 0$ for all $i\in\{1,\ldots,N\}$, implying $f(a_{i})=P^{*}$ for all $i\in\{1,\ldots,N\}$.\\

(iv) Using  (iii) we have that $L^{*}(f)=P^{*}=f(x^{*})$, and continuing with the same notation as in the proof of (iii) we got by unicity of the minimizer $x^{*}$, that $a_{i}=x^{*}$ for all $i\in\{1,\ldots,N\}$.
This implies that $L^{*}=\ev_{x^{*}}\text{ on }\R[\underline{X}]^{*}_{l}$, and evaluating in the polinomials $X_{1},\ldots,X_{N}\in\R[\underline{X}]_{1}$ we got that:
\begin{center}
$L^{*}(X_{i})=\ev_{x^{*}}(X_{i})=x_{i}^{*}$  for all $i\in\{1,\ldots,N\}$.
\end{center}
That is to say, $x^{*}=(L^{*}(X_{1}),\ldots,L^{*}(X_{n}))$.
\end{proof}


We can now reformulate our problem as:\\
\\
Given $d\in\N_{0}$ and $L\in\R[\underline{X}]^{*}_{2d+2}$ such that $L(\sum\R[\underline{X}]^{2}_{d+1})\subseteq \R_{\geq 0}$, we would like to obtain for all $p\in\R[\underline{X}]_{2d+2}$:

\begin{itemize}
\item Nodes $x_{1},\ldots,x_{r}\in\R^{n}$ and  and weights $\lambda_{1},\ldots,\lambda_{r}>0$ such that:
\begin{equation*}
L(p)=\sum_{i=1}^{r}\lambda_{i}p(x_{i})
\end{equation*}
\end{itemize}

in other words:

\begin{itemize}
\item $x_{1,1},\ldots,x_{1,n},\ldots,x_{r,1},\ldots,x_{r,n}\in\R$ and $a_{1},\ldots,a_{r}\in\R$ such that:
\begin{equation*}
L(p)=\sum_{i=1}^{r}a_{i}^{2}p(x_{i,1},\ldots,x_{i,n})
\end{equation*}
\end{itemize}

again with other words:
\begin{itemize}
\item $x_{1,1},\ldots,x_{1,n},\ldots,x_{r,1},\ldots,x_{r,n}\in\R$ and $a_{1},\ldots,a_{r}\in\R$ such that:
\begin{equation*}
\begin{aligned}
&L(p)=&&\Bigg\langle{\left(
\begin{array}{lll}
   p(x_{1,1},\ldots,x_{1,n})  &   &   \\
      & \ddots   &    \\
     &   &  p(x_{r,1},\ldots,x_{r,n}) \\

    \end{array}
   \right)\left(
 \begin{array}{l}
 a_{1}\\
 \vdots\\
 a_{n}\\    
 \end{array} 
 \right), \left(
 \begin{array}{l}
 a_{1}\\
 \vdots\\
 a_{n}\\    
 \end{array} 
\right)
    } \Bigg \rangle\\
\end{aligned}    
\end{equation*}
\end{itemize}

again written differently:
\begin{itemize}
\item $x_{1,1},\ldots,x_{1,n},\ldots,x_{r,1},\ldots,x_{r,n}\in\R$ and $a\in\R^{r}$ such that:

\end{itemize}

\begin{equation*}
\begin{aligned}
&L(p)=&&\left\langle{p\left[
\begin{pmatrix}
   x_{1,1}  &   &   \\
      & \ddots   &    \\
     &   & x_{r,1}
    \end{pmatrix},...,
    \begin{pmatrix}
   x_{1,n}  &   &   \\
      & \ddots   &    \\
     &   &  x_{r,n}
    \end{pmatrix}
   \right]
   \left(
 \begin{array}{l}
 a_{1}\\
 \vdots\\
 a_{n}\\    
 \end{array} 
 \right), \left(
 \begin{array}{l}
 a_{1}\\
 \vdots\\
 a_{n}\\    
 \end{array} 
\right)
    } \right\rangle\\
\end{aligned}    
\end{equation*}

again with less words:
\begin{itemize}
\item Diagonal matrices $D_{1},\ldots,D_{n}\in\R^{r\times r}$ and $a\in\R^{n}$ such that:
\begin{equation*}
L(p)=\langle p(D_{1},\ldots,D_{n})a,a\rangle
\end{equation*}
\end{itemize}

\begin{Reminder}\label{st}
Let $r,n\in\N$ and $M_{1},\ldots,M_{n}\in\R^{r \times r}$ symmetric commuting matrices. Then there exist an orthogonal matrix $P\in\R^{r\times r}$ such that $P^{t}M_{i}P$ is a diagonal matrix for all $ i\in\{1,\ldots,n\}$.
\end{Reminder}

Using this theorem  we can continue with our reformulation of the problem: given $d\in\N_{0}$ and  $L\in\R[\underline{X}]^{*}_{2d+2}$ such that $L(\sum\R[\underline{X}]^{2}_{d+1})\subseteq \R_{\geq 0}$, to find a quadrature rule for $L$ is the same as to find commuting symmetric matrices $M_{1},\ldots,M_{n}\in\R^{r\times r}$ and a vector $a\in\R^{r}$ such that:
\begin{equation}\label{newform}
L(p)=\langle p(M_{1},\ldots,M_{n})a,a\rangle
\end{equation}

We end the reformulation of the problem once and for all with the languages of endomorphisms, instead of matrices. That is to say: given $d\in\N_{0}$ and $L\in\R[\underline{X}]^{*}_{2d+2}$ such that $L(\sum\R[\underline{X}]^{2}_{d+1})\subseteq \R_{\geq 0}$, we would like to obtain a finite dimensional euclidean vector space $V$, commuting self-adjoint endomorphisms $M_{1},\ldots,M_{n}$ of $V$ and $a\in V$ such that:
\begin{equation}\label{newformope}
L(p)=\langle p(M_{1},\ldots,M_{n})a,a\rangle
\end{equation}

\begin{rem}\label{original}
Gelfand, Naimark and Segal gave a solution for the case we allow to the space $V$ to be infinite dimensional and the linear form to be stricly positive in the sums of squares, that is to say, in the case we are given a linear form $L\in\R[\underline{X}]^{*}$ such that $L(p^{2})>0$ for all $p\neq 0$. The solution was given by defining the inner product:
\begin{equation}\label{inner}
\langle p,q \rangle:=L(pq)
\end{equation}
and defining the self adjoint operators $M_{i}$, for all $i\in\{1,\ldots,n\}$, on the infinite dimensional vector  space $\R[\underline{X}]$,  in the following way:
\begin{center}
$M_{i}:\R[\underline{X}]\longrightarrow\R[\underline{X}]$, $p\mapsto X_{i}p$
\end{center}

Taking $a:=1\in\R[\underline{X}]$ we have the searched equality \eqref{newformope}.
\end{rem}

\fbox{\begin{minipage}{40em}
From now on we will assume we are given a linear form $L\in\R[\underline{X}]^{*}_{2d+2}$ for $d\in\N _{0}\ \cup \ \{\infty\}$ such that $L(\sum\R[\underline{X}]^{2}_{d+1})\subseteq\R_{\geq 0}$ or what is the same due to \ref{traduccion} and \ref{lecture} $M_{L}$ is positive semidefinite, unless $L$ is defined explicitely in other way.
\end{minipage}}

\section{truncated GNS-construction}
In this section we will explain how we can define the euclidean vector space and multiplications operators required in \eqref{newformope} from this positive semidefinite linear form $L$, in a similar way as in the Gelfand-Neimark-Segal construction \ref{original}.\\

First, we will get rid of the problem that, $L(p^{2})=0$ does not imply $p=0$ for every $p\in\R[\underline{X}]_{d+1}$, that is to say \eqref{inner} does not define an inner product if the linear form is positive semidefinite. By grouping together the polynomials with this property we will be able to define an inner product, on a quotient space. As a consequence, we will obtain an euclidean vector space. With respect to the multiplication operators, we will need to do the orthogonal projection on the class of polynomials with one degree less, in such a way that when we do the multiplication for the variable $X_{i}$ we are not out of our ambient space. This construction was already done in \cite{Put}.

\begin{dyn}
We define and denote the truncated GNS kernel of $L$:
\begin{equation}
U_{L}:=\{ p\in\R[\underline{X}]_{d+1}\text{ }|\text{ }L(pq)=0 \text{ for all } q\in\R[\underline{X}]_{d+1}   \}\nonumber
\end{equation}
\end{dyn}

\begin{prop}\label{nucleo0}
The truncated GNS kernel of $L$ is a vector subspace in $\R[\underline{X}]_{d+1}$. Moreover:
\begin{equation}\label{nucleo}
U_{L}=\{ p\in\R[\underline{X}]_{d+1}\text{ }|\text{ }L(p^{2})=0 \}
\end{equation}
\end{prop}

\begin{proof}
The fact that $U_{L}$ is a vector subspace follows directly from the linearity of $L$. Let us prove the equality \eqref{nucleo}. For this let us denote $A:=\{ p\in\R[\underline{X}]_{d+1}\text{ }|\text{ }L(p^{2})=0 \}$. The inclusion $U_{L}\subseteq A$ is trivial. For the other inclusion we will demonstrate first, due to $L$ is positive semidefinite and linear, that the Cauchy-Schwarz inequality holds:
\begin{equation}\label{cauchy}
L(pq)^{2}\leq L(p^{2})L(q^{2})
\end{equation}
Indeed, for all $t\in\R$ and $p,q\in\R[\underline{X}]_{d+1}$ it holds:
\begin{equation*}
0\leq L((p+tq)^{2})=L(p^{2})+2tL(pq)+t^{2}L(q^{2})
\end{equation*}
Therefore the polynomial $r:=L(p^{2})+2XL(pq)+X^{2}L(q^{2})\in\R[X]_{2}$ is non negative, i.e. $r(x)\geq 0$ for all $x\in\R$. In the case $L(q^{2})\neq 0$, the discriminant of  $r$ has to be less or equal to cero i.e. $4L(pq)^{2}-4L(p^{2}L(q^{2}))\leq 0$ and we get the searched inequality \eqref{cauchy}. In the case $L(q^{2})=0$, then $L(pq)=0$ and trivially we get also the inequality \eqref{cauchy}. As a consequence if $p\in A$ then $L(p^{2})=0$, and this implies due to \eqref{cauchy}, $L(pq)=0$ for all $q\in\R[\underline{X}]_{d}$, and therefore $p\in U_{L}$. 
\end{proof}

\begin{dyn}
We define and denote the GNS representation space of $L$, as the following quotient of vector spaces:
\begin{equation}
V_{L}:=\frac{\R[\underline{X}]_{d+1}}{U_{L}} 
\end{equation}
For every $p\in\R[\underline{X}]_{d+1}$ we will write $\overline{p}^{L}$ to refer us to the class of $p$ in $V_{L}$.
We define and denote the GNS inner product of $L$, in the following way:
\begin{equation}
\langle \overline{p}^{L},\overline{q}^{L}\rangle_{L}:=L(pq) 
\end{equation}
for every $p,q\in\R[\underline{X}]_{d+1}$.
\end{dyn}

\begin{prop}
$(V_{L},\langle\text{ . } ,\text{ . } \rangle_{L})$, is a symmetric bilinear form:
\end{prop}
\begin{proof}
Let us prove first that $\langle\text{ . } ,\text{ . } \rangle_{L}$ is well defined. To do this take $p_{1},q_{1},p_{2},q_{2}\in\R[\underline{X}]_{d+1}$ with $\overline{p_{1}}^{L}=\overline{p_{2}}^{L}$ and $\overline{q_{1}}^{L}=\overline{q_{2}}^{L}$ then:
\begin{equation*}
\begin{aligned}
\langle \overline{p_{1}}^{L},\overline{q_{1}}^{L} \rangle_{L}=\langle \overline{p_{2}}^{L},\overline{q_{2}}^{L} \rangle_{L} & \iff  L(p_{1}q_{1})=L(p_{2}q_{2}) \iff L(p_{1}q_{1})-L(p_{2}q_{2})=0&&\\
 &\iff L(p_{1}q_{1})+L(-p_{2}q_{1})-L(-p_{2}q_{1})-L(p_{2}q_{2})=0 &&\\
 & \iff L((p_{1}-p_{2})q_{1})-L(p_{2}(q_{2}-q_{1}))=0 &&\\
\end{aligned}
\end{equation*}
The last equality holds since $p_{1}-p_{2},q_{2}-q_{1}\in U_{L}$.
The bilinearity and symmetry is trivial. $\langle\text{ . } ,\text{ . } \rangle_{L}$ is positive semidefinite since $L(\sum\R[\underline{X}]_{d+1}^{2})\subseteq \R_{\geq 0}$. It remains to prove that $\langle\text{ . } ,\text{ . } \rangle_{L}$ is even positive definite. Indeed, for all $p\in\R[\underline{X}]_{d+1}$ with $\langle \overline{p}^{L},\overline{p}^{L} \rangle_{L}=0$ then $L(p^{2})=0$ and then $p\in U_{L}$ as we have shown in \ref{nucleo}. 
\end{proof}

\begin{dyn}\label{mo}
For $i\in\{1,\ldots,n\}$, we define the  $i$-th truncated GNS multiplication operator of $L$ as the following map between euclidean vector subspaces of $V_{L}$, and denote by $M_{L,i}$:
\begin{equation}
M_{L,i}:\Pi(V_{L})\longrightarrow\Pi_{L}(V_{L}),\text{  }\overline{p}^{L}\mapsto\Pi(\overline{pX_{i}}^{L})\text{ for } p\in\R[\underline{X}]_{d}
\end{equation}
where $\Pi_{L}$ is the orthogonal projection map of $V_{L}$ into the vector subspace $\{\text{ }\overline{p}^{L}\text{ }|\text{ }p\in\R[\underline{X}]_{d}\}$ with respect to the inner product $\langle\text{ . } ,\text{ . } \rangle_{L}$. We will call and denote the subvector vector space:
\begin{equation}
T_{L}:=\Pi(V_{L})=\{\text{ }\overline{p}^{L}\text{ }|\text{ }p\in\R[\underline{X}]_{d}\}
\end{equation}
of $V_{L}$, the  GNS-truncation of $L$.
\end{dyn}

\begin{prop}
The $i$-th truncated GNS multiplication operator of $L$ is a self-adjoint endomorphism of $T_{L}$.
\end{prop}
\begin{proof}
Let us demonstrate first that the $i$-th truncated GNS multiplication operator of $L$ is well defined. $M_{L,i}$ is well defined if and only if $M_{L,i}(\overline{p}^{L})=\overline{0}^{L}$ for all $p\in U_{L}\cap\R[\underline{X}]_{d} $ if and only if $\Pi_{L}(\overline{X_{i}p}^{L})=\overline{0}^{L}$ for all $p\in U_{L}\cap\R[\underline{X}]_{d}$. Since $\Pi_{L}(\overline{X_{i}p}^{L})\in T_{L}$  we can choose $q\in\R[\underline{X}]_{d}$ such that $\overline{q}^{L}=\Pi_{L}(\overline{X_{i}p}^{L})$ and then:
\begin{align}
L(q^2)=\langle \overline{q},\overline{q} \rangle_{L}&=\langle \Pi_{L}(\overline{X_{i}p}^{L}), \Pi_{L}(\overline{X_{i}p}^{L}) \rangle_{L} \stackrel{\Pi_{L}\circ\Pi_{L}=\Pi_{L}}{=}\nonumber\langle \Pi_{L}(\overline{X_{i}p}^{L}), \overline{X_{i}p}^{L} \rangle_{L}  \\ &=\langle \overline{q}^{L},\overline{X_{i}p}^{L} \rangle_{L} =L(q(X_{i}p))=L((qX_{i})p)\stackrel{p\in U_{L}}{=}0 \nonumber
\end{align}
Therefore $\Pi_{L}(\overline{X_{i}p}^{L})=\overline{0}^{L}$ for all $p\in U_{L}$. Let us see now that $M_{L,i}$ are self-adjoint endomorphisms, for this let $p,q\in\R[\underline{X}]_{d}$ then:

\begin{align}
\langle M_{L,i}\nonumber&(\overline{p}^{L}) , \overline{q}^{L} \rangle_{L}=
\langle \Pi_{L}(\overline{X_{i}p}^{L}), \overline{q}^{L} \rangle_{L}=\langle \overline{X_{i}p}^{L}, \Pi_{L}(\overline{q}^{L}) \rangle_{L}=\langle \overline{X_{i}p}^{L}, \overline{q}^{L} \rangle_{L}=L((X_{i}p)q)\\ \nonumber
 =L(p(&X_{i}q))=\langle \overline{p}^{L}, \overline{X_{i}q}^{L}\rangle_{L} =\langle \Pi_{L}(\overline{p}^{L}), \overline{X_{i}q}^{L}\rangle_{L}=\langle \overline{p}^{L}, \Pi_{L}(\overline{X_{i}q}^{L})\rangle_{L}=\langle \overline{p}^{L}, M_{L,i}(\overline{q}^{L})\rangle_{L}
   \nonumber
\end{align}
\end{proof}

\begin{rem}\label{GNSmodule}
The GNS construction for $L\in\R[\underline{X}]^{*}$ with $L(\sum\R[\underline{X}]^{2})\subseteq \R_{\geq 0}$ is the same as the original \eqref{original} modulo $U_{L}$. The GNS representation space of $L$ and the GNS truncation of $L$ are the same $\frac{\R[\underline{X}]}{U_{L}}$, where:
\begin{equation}
 U_{L}=\{p\in\R[\underline{X}]\text{ }|\text{ }L(p^{2})\geq 0\}
\end{equation}
The truncated GNS multiplication operators of $L$ commute, since $\frac{\R[\underline{X}]}{U_{L}}$ is a commutative ring. One can easily prove that $U_{L}$ is an ideal. Indeed it is clear that if $p,q\in U_{L}$ then $L((p+q)^{2})=0$, and if $p\in U_{L}$ and $q\in\R[\underline{X}]$ then $L(p^{2}q^{2})=L(p(pq^{2}))\stackrel{p\in U_{L}}{=}0$ implies $pq\in U_{L}$.
\end{rem}

\begin{nyl}\label{usefull}
Remember that $L':=L_{\R[\underline{X}]_{2d+1}}$. Let us denote as $B_{L}$ the transformation matrix of the following bilinear form with respect to the standard monomial basis:
\begin{center} 
$\R[\underline{X}]_{d+1}\times\R[\underline{X}]_{d}\longrightarrow \R, (p,q) \longmapsto L(pq) $
\end{center}
Then it holds $\rk M_{L'}=\rk B_{L} $ and for every such $L$ linear form we can define its respective modified moment matrix as:
\begin{equation}
\widetilde{M_{L}}:=
\left(
\begin{array}{c|c}
\makebox{$M_{L'}$}&\makebox{$M_{L'}W_{L}$}\\
\hline
  \vphantom{\usebox{0}}\makebox{$W^{T}_{L}M_{L'}$}&\makebox{$W^{T}_{L}M_{L'}W_{L}$}
\end{array}
\right)\nonumber
\end{equation}

where $W_{L}$ is a matrix such that $M_{L'}W_{L}=C_{L}$, where $C_{L}$ is the submatrix of $B_{L}$ remaining from eliminating the columns corresponding to the matrix $M_{L'}$. $\widetilde{M_{L}}$ is well defined since it does not depend from the election of $W_{L}$ and it is positive semidefinite.
\end{nyl}
\begin{proof}
Notice that $M_{L'}$ is the transformation matrix of the linear map:
\begin{center} 
$\varphi:=\R[\underline{X}]_{d}\longrightarrow \R[\underline{X}]_{d}^{*},p\mapsto(a\mapsto L(pq)) $
\end{center}
 with respect to the standard monomial basis and in the same way $B_{L}$ is the transformation matrix of the linear map:
 \begin{center}
   $\psi:=\R[\underline{X}]_{d+1}\longrightarrow\R[\underline{X}]^{*}_{d},p\mapsto(q\mapsto L(pq))$
   \end{center}
 with respect to the standard monomial basis. Note that to prove  $\rk M_{L'}=\rk B_{L}$ it is the same than to prove $\varphi(\R[\underline{X}]_{d}) = \psi(\R[\underline{X}]_{d+1})$. It is obvious that $\varphi(\R[\underline{X}]_{d})\subseteq  \psi(\R[\underline{X}]_{d+1})$. For the other inclusion we take $\Lambda\in\psi(\R[\underline{X}]_{d+1}) $ then there exits $p\in\R[\underline{X}]_{d+1}$ such that $\Lambda(q)=\psi(p)(q)=L(pq)$ for all $q\in\R[\underline{X}]_{d}$. We look for a $g\in\R[\underline{X}]_{d}$ such that $\Lambda(q)=L(gq)$ for all $q\in\R[\underline{X}]_{d}$, because then $\varphi(h)(q)=L(gq)=\Lambda(q)$ for all $q\in\R[\underline{X}]_{d}$ implying $\varphi(g)=\Lambda$ and then we could conclude $\Lambda\in\varphi(\R[\underline{X}]_{d}) $. In other words, we want to show that there exists $g\in\R[\underline{X}]_{d}$ such that : $L(pq)=L(gq)$ for every $q\in\R[\underline{X}]_{d}$. With more different words, our aim is to find $g\in\R[\underline{X}]_{d}$ such that:
 \begin{center}
 $\langle \overline{p},\overline{q} \rangle_{L}=\langle \overline{g},\overline{q} \rangle_{L}\text{ for every }q\in\R[\underline{X}]_{d}$
 \end{center}     
For this we define the following linear form:
 \begin{center}
 $\Lambda_{p}:=\frac{\R[\underline{X}]_{d}}{U_{L}\cap\R[\underline{X}]_{d}}\longrightarrow \R,\overline{q}\mapsto L(pq)\text{ for every q}\in\R[\underline{X}]_{d}$
 \end{center}          
$\Lambda_{p}\in(\frac{\R[\underline{X}]_{d}}{U_{L}\cap\R[\underline{X}]_{d}})^{*}$. Since $\frac{\R[\underline{X}]_{d}}{U_{L}}$ is a finite dimensional euclidean vector space is in particular a Hilbert space and then by the Fr\'{e}chet-Riesz  Representation Theorem there exists $\overline{g}\in\frac{\R[\underline{X}]_{d}}{U_{L}\cap\R[\underline{X}]_{d}}$ with $g\in\R[\underline{X}]_{d}$, such that:
\begin{center}
$\Lambda_{p}(\overline{q})=\langle \overline{q},\overline{g} \rangle_{L}\text{ for every } q\in\R[\underline{X}]_{d}$.
\end{center}                            
Therefore $L(pq)=\Lambda_{p}(\overline{q})=\langle \overline{q},\overline{g} \rangle=L(qg)$ for every $q\in\R[\underline{X}]_{d}$. Then $\rk M_{L'}=\rk B_{L}$  and therefore there exits $W_{L}$ (may not be unique) such that $M_{L'}W_{L}=C_{L}$. Now, we claim that the modified moment matrix $\widetilde{M_{L}}$, does not depend from the choice of the matrix $W_{L}$ with the property $M_{L'}W=C_{L}$. Indeed, assume there are matrices $W_{1},W_{2}$ such that $M_{L'}W_{1}=M_{L'}W_{2}$. Let us denote:
\begin{equation}
W_{1}:=(P_{1},\ldots,P_{r_{d+1}}),\text{ }W_{2}:=(Q_{1},\ldots,Q_{r_{d+1}})\nonumber
\end{equation}  
 where $P_{1},\ldots,P_{r_{d+1}}$ and $Q_{1},\ldots,Q_{r_{d+1}}$ are the  respective column vectors of the matrices $W_{1}$ and $W_{2}$ and define $p_{i}:=\poly(Q_{i})$ and $q_{i}:=\poly(Q_{i})$ for $i\in\{1,\ldots,r_{d+1}\}$. Then we have the following matrix equality:
\begin{center}
$M_{L'}(P_{1},\ldots,P_{r_{d+1}})=M_{L'}(Q_{1},\ldots,Q_{r_{d+1}})$ 
\end{center}  
Let $i\in\{1,\ldots,r_{d+1}\}$ then: 
\begin{equation}
M_{L'}P_{i}=M_{L'}Q_{i}\iff M_{L'}(P_{i}-Q_{i})=0\iff (P_{i}-Q_{i})^{T}M_{L'}(P_{i}-Q_{i})\stackrel[\eqref{nucleo}]{\ref{multiplicacion}}{\iff} p_{i}-q_{i}\in U_{L^{'}}\nonumber
\end{equation}
This implies $L'(p_{i}p_{j})=L'(q_{i}q_{j})$ and again due to \ref{multiplicacion} we get that:
\begin{center}
 $P_{i}^{T}M_{L^{'}}P_{j}=Q_{i}^{T}M_{L^{'}}Q_{j}\text{ for all }i,j\in\{1,\ldots,r\}$
 \end{center}
Then we have got that:
\begin{center}
$W_{1}^{T}M_{L^{'}}W_{1}=\left(
\begin{array}{lll}
  P_{1}^{T}M_{L^{'}}P_{1}  &\ldots   &   P_{1}^{T}M_{L^{'}}P_{r}   \\
     \vdots & \ddots   & \vdots    \\
    P_{r}^{T}M_{L^{'}}P_{1}   & \ldots  &  P_{r}^{T}M_{L^{'}}P_{r} \\

    \end{array}
   \right)=\left(
\begin{array}{lll}
  Q_{1}^{T}M_{L^{'}}Q_{1}  &\ldots   &   Q_{1}^{T}M_{L^{'}}Q_{r}   \\
     \vdots & \ddots   & \vdots    \\
    Q_{r}^{T}M_{L^{'}}Q_{1}   & \ldots  &  Q_{r}^{T}M_{L^{'}}Q_{r} \\

    \end{array}
   \right)=W_{2}^{T}M_{L^{'}}W_{2}$
\end{center}
Therefore $W_{1}^{T}M_{L'}W_{1}=W_{2}^{T}M_{L'}W_{2}$, and we can conclude that $\widetilde{M_{L}}$ is well defined.

Moreover since  $M_{L'}$ is a  is a positive semidefinite matrix, then there exists a matrix $C\in\R^{s_{d}\times s_{d}}$ such that $M_{L'}=CC^{T}$ due to \ref{mpsd}. Then  we have the following factorization:
\begin{equation}
\widetilde{M_{L}}=
\left(
\begin{array}{c|c}
\makebox{$CC^{T}$}&\makebox{$CC^{T}W$}\\
\hline
  \vphantom{\usebox{0}}\makebox{$W^{T}CC^{T}$}&\makebox{$W^{T}CC^{T}W$}
\end{array}
\right)=\left(
\begin{array}{c|c}
\makebox{$C$}&\makebox{$0$}\\
\hline
  \vphantom{\usebox{0}}\makebox{$W^{T}C$}&\makebox{$0$}
\end{array}
\right)\left(
\begin{array}{c|c}
\makebox{$C$}&\makebox{$0$}\\
\hline
  \vphantom{\usebox{0}}\makebox{$W^{T}C$}&\makebox{$0$}
\end{array}
\right)^{T}\nonumber
\end{equation}
then taking $P:=\left(
\begin{array}{c|c}
\makebox{$C$}&\makebox{$0$}\\
\hline
  \vphantom{\usebox{0}}\makebox{$W^{T}C$}&\makebox{$0$}
\end{array}
\right)$ 
we get that $\widetilde{M_{L}}=PP^{T}$, which due to \ref{mpsd}, proves $\widetilde{M_{L}}$ is positive semidefinite.
\end{proof}

\section{Gaussian quadrature rule}
In this section we will prove the existence of a quadrature rule representation for the positive semidefinite linear form $L$ on a set that cointains $\R[\underline{X}]_{2d+1}$ by providing that the truncated GNS multiplication operators commute. We will also demonstrate that this condition it is strictly more general than the very well known condition of being flat, condition that for its part ensure the existence of a quadrature rule representation for $L$ on the whole space in contrast with the quadrature rule in a space that contains $\R[\underline{X}]_{2d+1}$ that we get in case the truncated GNS multiplication operators commute.

\begin{prop}\label{dimension}
The vector spaces $T_{L}$ and $\frac{\R[\underline{X}]_{d}}{U_{L}\cap\R[\underline{X}]_{d}}$ are canonically isomorphic.

\end{prop}

\begin{proof}
Let us consider the following linear map between euclidean vector spaces:
\begin{equation}\label{sigmanull}
\sigma_{L}: T_{L}\longrightarrow V_{L'}:\overline{p}^{L} \longmapsto\overline{p}^{L'}\text{  for every }  p\in\R[\underline{X}]_{d}
\end{equation}
where remember we denoted $L':=L_{|\R[\underline{X}]_{2d}}$. It is well defined since for every $\overline{p}^{L}$, $\overline{q}^{L}\in T_{L}$ such that  $\overline{p}^{L}=\overline{q}^{L}$ we can assume without loss of generality that $p,q\in\R[\underline{X}]_{d}$, and therefore:

\begin{center}
$\overline{p}^{L}=\overline{q}^{L}\Leftrightarrow L((p-q)^{2})=0\Leftrightarrow L'((p-q)^2)=0\Leftrightarrow \overline{p}^{L'}=\overline{q}^{L'}\Leftrightarrow \sigma_{0}(\overline{p}^{L})=\sigma_{0}(\overline{q}^{L})$
\end{center}
 $\sigma_{0}$ is also a linear isometry, since for every $p,q\in\R[\underline{X}]_{d}$ we have:
\begin{equation*}
\langle \overline{p}^{L},\overline{q}^{L} \rangle_{L}=L(pq)=L^{'}(pq)=\langle \overline{p}^{L'},\overline{q}^{L'} \rangle_{L'}=\langle \sigma_{L}(\overline{p}^{L}), \sigma_{L}(\overline{q}^{L})\rangle_{L'}
\end{equation*} 
Then $\sigma_{L}$ is immediately injective. On other side, $\sigma_{L}$ is surjective since for every $\overline{p}^{L'}\in V_{L'}$ with $p\in\R[\underline{X}]_{d}$, it holds that $\sigma_{L}(\overline{p}^{L})=\overline{p}^{L'}$. Thence $\sigma_{L}$ is an isomorphism between vector spaces.
\end{proof}

\begin{notation}
For a linear form $\ell\in\R[\underline{X}]^{*}_{2d}$ such that $\ell(\sum\R[\underline{X}]_{d+1}^2)\subseteq\R_{\geq 0}$ we will detone by $\sigma_{\ell}$  the following isomorphism of euclidean vector spaces already defined in $\eqref{sigmanull}$:
\begin{equation}
\sigma_{\ell}:T_{\ell}\longmapsto V_{\ell'},\overline{p}^{\ell}\mapsto\overline{p}^{\ell'},\text{ for }p\in\R[\underline{X}]_{d} 
\end{equation} 
\end{notation}

\begin{rem}\label{auxiliar1}
For $v_{1},\ldots,v_{r}\in\R[\underline{X}]_{d}$, we have  $\overline{v_{1}}^{L},\ldots,\overline{v_{r}}^{L}$ is an orthonormal basis  of $T_{L}$ if and only if $ \overline{v_{1}}^{L'},\ldots,\overline{v_{r}}^{L'} $ is an orthonormal basis of $V_{L'}$. 
\end{rem}

The following Theorem and Lemma, are probably very well known and we will use them to prove Proposition \eqref{zero}. The proofs can be seen for example in \cite{cox} and \cite{lau}.

\begin{teor}\label{lio}
And ideal $I\subseteq \R[\underline{X}]$ is zero dimensional (i.e. $|V_{\C}(I)|<\infty$) if and only if the vector space $\R[\underline{X}]/I$ is finite dimensional. Moreover $|V_{\C}(I)|\leq\dim(\R[\underline{X}]/I)$, with equality if and only if the ideal $I$ is radical.
\end{teor}
\begin{proof}
Theorem $2.6$ page $15$ in \cite{lau}.
\end{proof}

\begin{defi}
Let $I\subseteq\R[\underline{X}]$ be an ideal. $I$ is said to be radical when $I=\mathcal{I}(V_{\C}(I))$.
\end{defi}

\begin{lem}\label{radicalillo}
Let  $I\subseteq \R[\underline{X}]$ be an ideal. $I$ is  radical if and only if 
\begin{equation}\label{radical}
\text{For all } g\in\R[\underline{X}]\text{ such that }g^{2}\in I{\implies}g\in I
\end{equation}
\end{lem}
\begin{proof}
 There is a proof in \cite[Lemma 2.2]{lau}.
\end{proof}
\begin{prop}\label{flor}
Let $\Lambda\in\R[\underline{X}]^{*}$ such that $\Lambda(\sum\R[\underline{X}]^{2})\subseteq\R_{\geq 0}$. Then $U_{\Lambda}$ is a radical ideal.
\end{prop}
\begin{proof}
In \ref{GNSmodule} we saw that $U_{\Lambda}$ is and ideal, let us prove that it is real radical ideal. Let $g\in\R[\underline{X}]$ such that  $g^{2}\in U_{\Lambda}$. In particular $\Lambda(g^{2}1)=0$ and this implies $g\in U_{\Lambda}$.
\end{proof}

\begin{prop}\label{zero}
Let $\Lambda=\sum_{i=1}^{N}\lambda_{i}\ev_{a_{i}}\in\R[\underline{X}]^{*}$, with $N\in\N$, $\lambda_{1}>0,\ldots,\lambda_{N}>0$, and $a_{1},\ldots,a_{N}\in\R^{n}$ then:
\begin{equation}
  \dim(\frac{\R[\underline{X}]}{U_{\Lambda}})=|\{a_{1},\ldots,a_{N}\}| \nonumber
\end{equation}
\end{prop}
\begin{proof}
We have the following equalities:
\begin{align}
U_{\Lambda}  &\nonumber=\{ p\in\R[\underline{X}]\text{ }|\sum_{i=1}^{N}\lambda_{i}p^{2}(a_{i})=0\}=
\text{  }\{p\in\R[\underline{X}]\text{ }|\text{  }p^{2}(a_{i})=0\text{ for all }i\in\{1,\ldots,n\} \} \\ \nonumber
&\nonumber=\{p\in\R[\underline{X}]\text{ }|\text{  }p(a_{i})=0\text{ for all }i\in\{1,\ldots,n\} \}=
\it{I}(\{a_{1},\ldots,a_{N}\}) \nonumber
\end{align}
 
and since $\{a_{1},\ldots,a_{N}\}\subseteq\R^{n}$ is an algebraic set, by the ideal-variety correspondence (see \cite{cox}), it holds: 
\begin{center}
 $V_{\C}(\it{I}(\{a_{1},\ldots,a_{N}\}))=\{a_{1},\ldots,a_{N}\}$
\end{center}
what is the same as $V_{\C}(U_{\Lambda})=\{a_{1},\ldots,a_{N}\}$. Notice that by Theorem \eqref{lio} is enough to prove that $U_{\Lambda}$ is radical to finish the proof. In fact by Proposition \ref{flor} $U_{\Lambda}$ is radical.  Applying Theorem \ref{lio} we have the result.
\end{proof}


Let us review some known bounds on the number of nodes of quadrature rules for $L$ on $\R[\underline{X}]_{2d+2}$ and on $\R[\underline{X}]_{2d+1}$ (see \cite{cf0} and \cite{Put}).

\begin{prop}\label{extremalbound}
Then number of nodes $N$, of a quadrature rule for $L$ satisfies:
\begin{equation}
\rk M_{L}\leq N\leq |V_{\C}(U_{L})|\nonumber
\end{equation}
\end{prop}
\begin{proof}
Let $L=\sum_{i=1}^{N}\lambda_{i}\ev_{a_{i}}\in\R[\underline{X}]^{*}_{2d+2}$ for $a_{i},\ldots,a_{N}\in\R^{n}$ pairwise different points and $\lambda_{1},\ldots,\lambda_{N}>0$ weights and define $\Lambda:=\sum_{i=1}^{N}\lambda_{i}\ev_{a_{i}}\in\R[\underline{X}]^{*}$. Let us consider the following canonical map:
\begin{equation}
\frac{\R[\underline{X}]_{d+1}}{U_{L}}\hookrightarrow\frac{\R[\underline{X}]}{U_{\Lambda}}\nonumber
\end{equation}
By Proposition \ref{zero} we have that:
\begin{equation}
 \rk M_{L} = \dim(\frac{\R[\underline{X}]_{d+1}}{U_{L}})\leq \dim(\frac{\R[\underline{X}]}{U_{\Lambda}})=N \nonumber
\end{equation} 
On the other side, it holds that $\{a_{1},\ldots,a_{N} \}\subseteq V_{\C}(U_{L})$, since for all $p\in U_{L}$ we have $L(p^{2})=0$ and then $p(a_{i})=0$ for all $i\in\{1,\ldots,N\}$. this implies $ N\leq |V_{\C}(U_{L})|$.
\end{proof}

\begin{prop}\label{fbound}
The number of nodes $N$, of a quadrature rule for $L$ on $\R[\underline{X}]_{2d+1}$ satisfies:
\begin{equation}
 N\geq \dim (T_{L}) \nonumber
\end{equation}
\end{prop}
 \begin{proof}
Assume that $L$ has a quadrature rule on $\R[\underline{X}]_{2d+1}$ such that:
\begin{equation}
 L(p)=\sum_{i=1}^{N}\lambda_{i}p(a_{i})\nonumber
\end{equation} 
for every $p\in\R[\underline{X}]_{2d+1}$, where we can assume without loss of generality that the points $a_{1},\ldots,a_{N}\in\R^{n}$ are pairwise different and $\lambda_{1},\ldots,\lambda_{N}>0$ with $N<\infty$ for $N\in\N$. Let us set $\Lambda:=\sum_{i=1}^N\lambda_{i}\ev_{a_{i}}\in\R[\underline{X}]^{*}$. Then, the following linear map between euclidean vector spaces is an isometry:
\begin{equation}\label{sigma1}
\sigma_{1}:T_{L}\longrightarrow \frac{\R[\underline{X}]}{U_{\Lambda}}, \overline{p}^{L}\mapsto\overline{p}^{\Lambda}\text{ for every }p\in\R[\underline{X}]_{d}
\end{equation}

It is easy to see that is well defined since $U_{L}\subseteq U_{\Lambda}$. It holds also that $\sigma_{1}$ is  a linear isometry since, for all $p,q\in\R[\underline{X}]_{d}$:
\begin{center}
$\left\langle \overline{p}^{L},\overline{q}^{L}\right\rangle_{L}=L(pq)=\Lambda(pq)=\left\langle \overline{p}^{\Lambda},\overline{q}^{\Lambda} \right\rangle_{\Lambda}=\langle \sigma_{1}(\overline{p}^{L}),\sigma_{1}(\overline{q}^{L}) \rangle_{\Lambda}$
\end{center} 
Since, $\sigma_{1}$ is a linear isometry is inmmediately injective, and then:
\begin{center}
 $\dim (T_{L})\leq \dim (\frac{\R[\underline{X}]}{U_{\Lambda}})$
 \end{center}
And now we can apply the Proposition \ref{zero}, to conclude the proof.
\end{proof}

\begin{defi}\label{gaussian}
A quadrature rule for $L$ on $\R[\underline{X}]_{2d+1}$  with minimal number of nodes, that is to say with $\dim (T_{L})$ nodes is called a Gaussian quadrature rule. 
\end{defi} 

\begin{lem}\label{fields}
Assume that the truncated multiplication operators commute. Then for all $p\in\R[\underline{X}]_{d+1}$ we have the following equality:
\begin{equation}\label{key}
p(M_{L,1},\ldots,M_{L,n})(\overline{1}^{L})=\Pi_{L}(\overline{p}^{L})
\end{equation}
\end{lem}
\begin{proof}
Let $p=\underline{X}^{\alpha}$ for $\alpha\in\N^{n}$ with $|\alpha|\leq d+1$. We continue the proof by induction on $|\alpha|$:
\begin{itemize}
\item For  $|\alpha|=0$, we have that $\underline{X}^{\alpha}=1$  then:
\begin{center}
$ 1(M_{L,1},\ldots,M_{L,n})(\overline{1}^{L})=\id_{V_{L}}(\overline{1}^{L})=\overline{1}^{L}=\Pi_{L}(\overline{1}^{L})$
\end{center}
\item Let assume the statement is true for $|\alpha|=d$. Let us show it is also true for $|\alpha|=d+1$. Let $p=X_{i}q$ for some $i\in\{1,\ldots,n\}$ and $q=\underline{X}^{\beta}$ with $|\beta|=d$, then
$\Pi_{L}(\overline{q}^{L})=\overline{q}^{L}$ since $\overline{q}^{L}\in T_{L}$, and we have:
\begin{align}
p(M_{L,1},\ldots,M_{L,n})(\overline{1}^{L})&\nonumber=(M_{L,i}\circ q(M_{L,1},\ldots,M_{L,n}))(\overline{1}^{L})=\\ \nonumber
M_{L,i}(q(M_{L,1},\ldots,M_{L,n})(\overline{1}^{L}))&=M_{L,i}(\overline{q}^{L})=\Pi_{L}(\overline{X_{i}q}^{L})=\Pi_{L}(\overline{p}^{L}) \nonumber
\end{align}
\end{itemize}
since we have proved \eqref{key} for monomials then by the linearity of the orthogonal projection \eqref{key} is also true for polynomials.
\end{proof}

\begin{teor}\label{importante}
Assume the truncated multiplication operators of $L$ commute, and consider the set:
\begin{equation}\label{semiraro}
G_{L}:=\{\sum_{i=1}^{s}p_{i}q_{i}\text{ }|\text{ }s\in\N\text{, }p_{i}\in\R[\underline{X}]_{d+1}\text{ and } q_{i}\in\R[\underline{X}]_{d}+U_{L} \} 	
\end{equation}
then there exists a quadrature rule for $L$ on $G_{L}$ with $\dim(T_{L})$ many nodes.
\end{teor}
\begin{proof}
Since the truncated multiplication operators of $L$ commute by the Remeinder \ref{st} there exists an orthonormal basis $v:=\{v_{1},\ldots,v_{N}\}$ of $T_{L}$ consisting of common eigenvectors of the GNS truncated multiplation operators of $L$. That is to say, there exist $a_{1},\ldots,a_{N}\in\R^{n}$ such that:
\begin{center}
$M_{L,i}v_{j}=a_{j,i}v_{j}$ for all $i\in\{1,\ldots,n\}$ and $j\in\{1,\ldots,N\}$
\end{center}
 where $N:=\dim(T_{L})$. Since it always holds $\overline{1}^{L}\in T_{L}$ since $d\in\N$, we can write:
\begin{equation}\label{pesos}
 \overline{1}^{L}=b_{1}v_{1}+\cdots+b_{N}v_{N}
\end{equation}
for some $b_{1},\ldots,b_{N}\in\R$. Let us define $\lambda_{i}:=b_{i}^{2}$ for all $i\in\{1,\ldots,N\}$. 
Let $g=pq$ such that $p\in\R[\underline{X}]_{d+1}$ and $q\in\R[\underline{X}]_{d}+U_{L}$, then using Lemma \ref{fields} we have the two equalities: 
\begin{equation}\label{madremia}
\Pi_{L}(\overline{p}^{L})=p(M_{L,1},\ldots,M_{L,n})(\overline{1}^{L})\text{ and }\overline{q}^{L}=q(M_{L,1},\ldots,M_{L,n})(\overline{1}^{L})
\end{equation}
Using this equalities \eqref{madremia}, using that the orthogonal projection $\Pi_{L}$ is selfadjoint, using that $\{v_{1},\ldots,v_{N}\}$ is an orthonormal basis of $T_{L}$ consisting of common eigenvectors  of the GNS truncated multiplication operators of $L$ and also using the equation \eqref{pesos}, with the same idea as we got the reformulation of the problem in \eqref{newformope} we have:

\begin{align}
L(g)\nonumber =L(pq)=\langle \overline{p}^{L},\overline{q}^{L} \rangle_{L}\stackrel{\overline{q}^{L}\in T_{L}}{=}
\langle \overline{p}^{L},\Pi_{L}(\overline{q}^{L})\rangle_{L}=& \\ \nonumber 
\langle \Pi_{L}(\overline{p}^{L}),\overline{q}^{L} \rangle_{L} 
=\sum_{j=1}^{N}b_{j}^{2}p(a_{j})q(a_{j})=\sum_{j=1}^{N}\lambda_{j}p(a_{j})q(a_{j})&
\end{align}
Then by linearity it holds that $L(p)=\sum_{i=1}^{N}\lambda_{i}p(a_{i})$ for all $p\in G_{L}$. It remains to prove that the nodes of the quadrature rule for $L$ that we got, $a_{1},\ldots,a_{N}\in\R^{n}$ are pairwise different, but this is true since $N=\dim T_{L}$ is the minimal possible number of nodes for a quadrature rule on $\R[\underline{X}]_{2d+1}$ as we proved in \ref{fbound}.
 
\end{proof}

\begin{rem}\label{particular}
Since $\R[\underline{X}]_{2d+1}\subseteq G_{L}$, in the conditions of Theorem \ref{importante} we got in particular a Gaussian quadrature rule  for the linear form $L$.
\end{rem}

\begin{cor}
Let $n=1$, i.e. $L\in\R[X]^{*}_{2d+2}$ with $L(\sum\R[X]^{2})\geq 0$. Then $L$ has a quadrature rule on $G_{L}$ \eqref{semiraro}.
\end{cor}
\begin{proof}
$L$ has one truncated GNS multiplication operator, therefore the hypothesis of Theorem \ref{importante} holds and there is a quadrature rule on $G_{L}$ for $L$.
\end{proof}

\begin{prop}\label{flat}
The following assertions are equivalent:
\begin{enumerate}[label=(\roman*)]
\item $\R[\underline{X}]_{d+1}=\R[\underline{X}]_{d}+U_{L}$
\item $T_{L}=V_{L}$
\item For all $\alpha\in\N_{0}^{n}$ with $|\alpha|=d+1$, there exists $p\in\R[\underline{X}]_{d}$ such that $\underline{X}^{\alpha}-p\in U_{L}$
\item The canonical map:
\begin{equation}\label{canonical}
V_{L^{'}}=\R[\underline{X}]_{d}/ U_{L^{'}}\hookrightarrow \R[\underline{X}]_{d+1} /U_{L}=V_{L} 
\end{equation}
is an isomorphism.
\item $\dim(V_{L^{'}})=\dim(V_{L})$
\item The moment  matrices $(L(\underline{X}^{\alpha+\beta}))_{|\alpha|,|\beta|\leq d}$ and $(L(\underline{X}^{\alpha+\beta}))_{|\alpha|,|\beta|\leq d+1}$ have the same rank.
\item $M_{L}=\widetilde{M_{L}}$.
\end{enumerate}
\end{prop}
\begin{proof}
Note that the map \eqref{canonical} it is well defined since $\R[\underline{X}]_{d}\cap U_{L}=U_{L'}$. And one can see inmediately that:
\begin{center}
$(i)\iff(ii)\iff(iii)\iff(iv)\iff(v)$.
\end{center}

Let us show $(v)\iff(vi)$: $(L(\underline{X}^{\alpha+\beta}))_{|\alpha|,|\beta|\leq d+1}$ is the transformation matrix (or the associated matrix) of the bilinear form:
\begin{center}
$\R[\underline{X}]_{d+1}\times\R[\underline{X}]_{d+1} \longrightarrow \R$,$(p,q)\mapsto L(pq)$
\end{center}
with respect to the the standard monomial basis, and therefore it is also the transformation matrix (or the associated matrix) of the linear map:
\begin{equation}\label{lumin}
\R[\underline{X}]_{d+1}\longrightarrow\R[\underline{X}]^{*}_{d+1},p\mapsto (q\mapsto L(pq))
\end{equation}
with respect to the corresponding dual basis of the standard monomial basis. The kernel of this linear map \eqref{lumin} is $U_{L}$, in consequence:
\begin{center}
$\rk((L(\underline{X}^{\alpha+\beta}))_{|\alpha|,|\beta|\leq d+1})=\dim\R[\underline{X}]_{d+1}-U_{L}=\dim V_{L}$
\end{center}
reasoning in the same way:
\begin{equation}
\rk((L(\underline{X}^{\alpha+\beta}))_{|\alpha|,|\beta|\leq d})=\dim V_{L^{'}}\nonumber
\end{equation}
Finally $(vi)\iff(vii)$:
\begin{align}
\rk(&(L(\underline{X}^{\alpha+\beta}))_{|\alpha|,|\beta|\leq d})\iff \rk(M_{L'})=\rk(M_{L}) \nonumber\\ \nonumber
\iff&\left(
\begin{array}{c|c}
\makebox{$M_{L'}$}&\makebox{$M_{L'}W_{L}$}\\
\hline
  \vphantom{\usebox{0}}\makebox{$W^{T}_{L}M_{L'}$}&\makebox{$W^{T}_{L}M_{L'}W_{L}$}
\end{array}
\right)=\left(
\begin{array}{c|c}
\makebox{$M_{L'}$}&\makebox{$M_{L'}W_{L}$}\\
\hline
  \vphantom{\usebox{0}}\makebox{$W^{T}_{L}M_{L'}$}&\makebox{$C_{L}$}
\end{array}
\right) \\ \nonumber
\iff& W_{L}M_{L'}W_{L}=C_{L}\iff\widetilde{M_{L}}=M_{L} \\ \nonumber
\end{align}
\end{proof}

\begin{defi}\label{definiciones}
We say the linear form $L$ is flat if the conditions  $(i)$, $(ii)$, $(iii)$, $(iv)$, $(v)$, $(vi)$ and $(vii)$  in \eqref{flat} are satisfied.
\end{defi}

\begin{prop}\label{importante2}
Suposse $L$ is flat then the truncated GNS operators of $L$ commute
\end{prop}

\begin{proof}
Asumme $L$ is flat, and let $i,j\in\{1,\ldots,n\}$ and $p\in\R[\underline{X}]_{d}$ . We want to prove:
\begin{equation}
M_{L,i}\circ M_{L,j}(\overline{p}^{L})= M_{L,j}\circ M_{L,i}(\overline{p}^{L})\nonumber
\end{equation}
Let us write $X_{i}p=p_{1}+q_{1}$ and $X_{j}p=p_{2}+q_{2}$ with $p_{1},p_{2}\in \R[\underline{X}]_{d}$
and $q_{1},q_{2}\in U_{L}$. Then $M_{L,j}(\overline{p}^{L})=\Pi_{L}(\overline{X_{j}p}^{L})=\Pi_{L}(\overline{p_{2}+q_{2}}^{L})=\Pi_{L}(\overline{p_{2}}^{L})+\Pi_{L}(\overline{q_{2}}^{L})=\Pi_{L}(\overline{p_{2}}^{L})=\overline{p_{2}}^{L}$. In the same way we get $M_{L,i}(\overline{p}^{L})=\overline{p_{1}}^{L}$. Therefore:
\begin{equation*}
\begin{aligned}
M_{L,i}\circ M_{L,j}(\overline{p}^{L})= M_{L,j}\circ M_{L,i}(\overline{p}^{L})\Longleftrightarrow & M_{L,i}(\overline{p_{2}}^{L})=M_{L,j}(\overline{p_{1}}^{L})&&\\
\Longleftrightarrow\Pi_{L}(\overline{X_{i}p_{2}}^{L})=\Pi_{L}(\overline{X_{j}p_{1}}^{L})&
&&\\
\end{aligned}
\end{equation*}
In other words, define $\overline{g}^{L}:=\Pi_{L}(X_{i}p_{2}-X_{j}p_{1})\in T_{L}$ for some $g\in\R[\underline{X}]_{d}$, then it is enough to show $g\in U_{L}$. Indeed:
\begin{align}
L\nonumber &(g^{2})=\langle \Pi_{L}(\overline{X_{i}p_{2}-X_{j}p_{1}}),\overline{g}^{L} \rangle_{L}=\langle \overline{X_{i}p_{2}-X_{j}p_{1}},\Pi_{L}(\overline{g}^{L})\rangle_{L} =\langle \overline{X_{i}p_{2}-X_{j}p_{1}},\overline{g}\rangle_{L}\\ \nonumber
=&L((X_{i}p_{2}-X_{j}p_{1})g)=L((X_{i}g)p_{2})-L((X_{j}g)p_{1})= 
\langle \overline{ X_{i}g}^{L},\overline{p_{2}}^{L}\rangle_{L}-\langle\overline{ X_{j}g}^{L},\overline{p_{1}}^{L} \rangle_{L} \\ \nonumber
&\stackrel[\overline{X_{i}p}^{L}=\overline{p_{1}}^{L}] {\overline{X_{j}p}^{L}=\overline{p_{2}}^{L}}{=}\langle\overline{ X_{i}g}^{L},\overline{X_{j}p}^{L} \rangle_{L}-\langle \overline{ X_{j}g}^{L},\overline{X_{i}p}^{L} \rangle_{L}=L(X_{i}gX_{j}p)-L(X_{j}gX_{i}p)=0
\end{align}
\end{proof}

Here we show some examples which shows that the reverse of Proposition \ref{importante2} does not hold.

\begin{ej}\label{elprimero}
The truncated GNS multiplication operators of the following linear form:
\begin{equation}
 L:\R[X_{1},X_{2}]_{4}\rightarrow \R,  p \mapsto \frac{1}{4}(p(0,0)+p(1,0)+p(-1,0)+p(0,1))\nonumber
\end{equation}
commute but $L$ is not flat. Indeed, if we do the truncated GNS-construction we have:

\begin{equation}
M_{L} = \begin{blockarray}{ccccccc}
\text{ } & 1 & X_{1} & X_{2} & X_{1}^{2} & X_{1}X_{2} & X_{2}^{2}\\
\begin{block}{c(ccc|ccc)}
1&1 & 0 & \frac{1}{4} & \frac{1}{2} & 0 & \frac{1}{4} \\
X_{1}&0 & \frac{1}{2} & 0 & 0 & 0 & 0 \\
X_{2} & \frac{1}{4} & 0 & \frac{1}{4} & 0 & 0  & \frac{1}{4}\\ \cline{2-7}
X_{1}^{2}&\frac{1}{2} & 0 & 0 & \frac{1}{2} & 0 & 0\\
X_{1}X_{2} & 0 & 0 & 0 & 0 & 0 & 0 \\
X_{2}^{2} & \frac{1}{4} & 0 & \frac{1}{4} & 0 & 0 & \frac{1}{4}\\
\end{block}
\end{blockarray}=\left(
\begin{array}{c|c}
\makebox{$A_{L}$}&\makebox{$B_{L}$}\\
\hline
  \vphantom{\usebox{0}}\makebox{$B_{L}^{T}$}&\makebox{$C_{L}$}
\end{array}\right)\nonumber
\end{equation}
is the associated moment matrix of the linear form $L$ and a basis of the truncated GNS-kernel of $L$ is $\left\langle X_{1}X_{2},X_{2}^{2}-X_{2}\right\rangle$. That is, the rank of $M_{L}$ is 4. And since in the kernel there is no polynomials of degree  less or equal to 1, we get that the unique element in the kernel of $L'$ is $0 $, then the truncated GNS space is $\frac{\R[X_{1},X_{2}]_{1}}{U_{L'}}\cong \R[X_{1},X_{2}]_{1}$, which implies the dimension of the GNS-truncated space is 3 and therefore $L$ is not flat by $(vi)$ in \ref{flat}. We can also verify that $L$ is not flat by computing $\widetilde{M_L}$. Indeed, in this case $A_{M}$ is invertible and $W_{M}$ is uniquely defined by $W_M=A_{M}^{-1}B_{M}$, then $\widetilde{M_{L}}$ reads:

\begin{equation}
\widetilde{M_{L}} = \left(
\begin{array}{ccc|ccc}
1 & 0 & \frac{1}{4} & \frac{1}{2} & 0 & \frac{1}{4} \\
0 & \frac{1}{2} & 0 & 0 & 0 & 0 \\
\frac{1}{4} & 0 & \frac{1}{4} & 0 & 0  & \frac{1}{4}\\ \cline{1-6}
\frac{1}{2} & 0 & 0 & \frac{\textbf{1}}{\textbf{3}} & 0 & 0\\
0 & 0 & 0 & 0 & 0 & 0 \\
\frac{1}{4} & 0 & \frac{1}{4} & 0 & 0 & \frac{1}{4}
\end{array}
\right)\nonumber
\end{equation}

Since $\widetilde{M_{L}}\neq M_{L}$ then $L$ is not flat by $(vii)$ in \ref{flat}.\\

Let us compute the truncated GNS multiplication operators of $L$. First note that:
\begin{equation}
 T_{L}\cong\frac{\R[X_{1},X_{2}]_{1}}{U_{L'}}=\left\langle \overline{1}^{L'},\overline{X_{1}}^{L'},\overline{X_{2}}^{L'} \right\rangle\nonumber
\end{equation}
Therefore by Remark \ref{auxiliar1} the truncated GNS space of $L$ is:
\begin{equation}
T_{L}=\left\langle \overline{1}^{L},\overline{X_{1}}^{L},\overline{X_{2}}^{L} \right\rangle\nonumber
\end{equation}
With the Gram-Schmidt orthonormalization process we get the following orthonormal basis with respect to the GNS product of $L$:
\begin{equation*}
 \underline{v} := \{ \overline{1}^{L}, \overline{\sqrt{2}X_{1}}^{L}, \overline{-\frac{\sqrt{3}}{3} + \frac{4\sqrt{3}}{3}X_{2}}^{L} \}
 \end{equation*}
 The matrices of the GNS-multiplication operators with respect to this orthonormal basis are:
$$
A_{1}:=M(M_{L,X_{1}},\underline{v}) = \left(
\begin{array}{ccc}
0 & \frac{\sqrt{2}}{2} & 0  \\
\frac{\sqrt{2}}{2}& 0 & -\frac{\sqrt{6}}{6} \\
0 & -\frac{\sqrt{6}}{6}  & 0 
\end{array}
\right)
$$

$$
A_{2}:=M(M_{L,X_{2}},\underline{v}) = \left(
\begin{array}{ccc}
\frac{1}{4} & 0 & \frac{\sqrt{3}}{4}  \\
0 & 0 & 0 \\
\frac{\sqrt{3}}{4} & 0  & \frac{3}{4} 
\end{array}
\right)
$$
It is easy to check that the truncated GNS multiplication operators of $L$ commute, that is $M_{L,X_{1}}\circ M_{L,X_{2}}-M_{L,X_{2}}\circ M_{L,X_{1}}=0$. Now since $M_{L,X_{1}}$ and $M_{L,X_{2}}$ commute we can do the simultaneous diagonalization on both of them, in order to find an orthonormal basis of the GNS truncation of $L$ consisting of common eigenvectors of $M_{L,X_{1}}$ and $M_{L,X_{2}}$. To do this we follow the same idea as in \cite[Algorithm 4.1, Step 1]{japoneses} and compute for a matrix:
\begin{center} 
$A=r_{1}A_{1}+r_{2}A_{2}$ where $r_{1}^{2}+r_{2}^{2}=1$
\end{center}
a matrix $P$ orthogonal such that $P^{T}AP$ is a diagonal matrix. In this case, we get for:
\begin{equation}
P=\left(\begin{array}{ccc}
   \frac{1}{2} & -\frac{\sqrt{6}}{4} & -\frac{\sqrt{6}}{4} \\ 
   0 & \frac{\sqrt{2}}{2}  &  -\frac{\sqrt{2}}{2}  \\
   \frac{\sqrt{3}}{2} & \frac{\sqrt{2}}{4} &  \frac{\sqrt{2}}{4}  \\ 
    \end{array}
    \right) \nonumber
\end{equation}
\begin{equation}
P^{T}A_{1}P=\left(\begin{array}{ccc}
   0 & 0 & 0 \\ 
   0 & -\frac{\sqrt{6}}{3}  &  0  \\
   0 & 0 &  \frac{\sqrt{6}}{3}  \\ 
    \end{array}
    \right)
		\text { and }
		P^{T}A_{2}P=\left(\begin{array}{ccc}
   1 & 0 & 0 \\ 
   0 & 0  &  0  \\
   0 & 0 &  0  \\ 
    \end{array}
    \right)\nonumber
		\end{equation}
Looking over the proof of \ref{importante} we can obtain the weights $\lambda_{1},\lambda_{2},\lambda_{3}\in\R_{>0}$ through the following operations:
\begin{equation}
P^{T}\left(\begin{array}{c}
   1  \\ 
   0 \\
   0  \\ 
    \end{array}
    \right)=\left(\begin{array}{c}
   \text{     }\frac{1}{2}\\ 
   -\frac{\sqrt{6}}{4}   \\
   -\frac{\sqrt{6}}{4}  \\ 
    \end{array}
    \right)\nonumber
\end{equation}
then $\lambda_{1}=(\frac{1}{2})^2$ and $\lambda_{2}=\lambda_{3}=(-\frac{\sqrt{6}}{4})^2$. Therefore we get the following decomposition: 
\begin{equation}
\widetilde{M}_{L}=\frac{1}{4}V_{2}(0,1)V_{2}(0,1)^{T}+\frac{3}{8}V_{2}(-\frac{\sqrt{6}}{3},0)V_{2}(-\frac{\sqrt{6}}{3},0)^{T}+\frac{3}{8}V_{2}(\frac{\sqrt{6}}{3},0)V_{2}(\frac{\sqrt{6}}{3},0)^{T}.\nonumber
\end{equation}
\end{ej}

\begin{ej}\label{volver}
Let us do the truncated GNS construction for the optimal solution that we got on the polynomial optimization problem described in \ref{porfavor}, that is:
\begin{equation}
\textbf{M}:=M_{2}(y) =  \begin{blockarray}{ccccccc}
\text{ } & 1 & X_{1} & X_{2} & X_{1}^2 & X_{1}X_{2} & X_{2}^2 \\
\begin{block}{r(rrr|rrr)}
 1 &   1.0000  &  0.7175  &  1.4698 &   0.5149  &  1.0547  &  2.1604  \\
 X_{1} &   0.7175  &  0.5149  &  1.0547 &   0.3694  &  0.7568  &  1.5502   \\
 X_{2} &  1.4698  &  1.0547  &  2.1604 &   0.7568  &  1.5502  &  3.1755  \\ \cline{2-7}
 X_{1}^2 &   0.5149  &  0.3694  &  0.7568 &   0.2651  &  0.5430  &  1.1123  \\
 X_{1}X_{2} &   1.0547  &   0.7568 &  1.5502 &   0.5430  &  1.1123 &   2.2785   \\
 X_{2}^2 &   2.1604  &  1.5502  &  3.1755 &   1.1123  &  2.2785  &  8.7737   \\
   \end{block}
   \end{blockarray}
\end{equation}
 Setting $\alpha:=\textbf{M}(1,2)$ and $\beta:=\textbf{M}(1,3)$, the truncated GNS kernel of M is:
\begin{equation}
U_{\textbf{M}}=\left\langle-\alpha+X_{1},-\beta+X_{2},-\alpha^{2}+X_{1}^2,-\alpha\beta+X_{1}X_{2}\right \rangle\nonumber
\end{equation}
the truncated GNS representation space is:
\begin{equation}
V_{\textbf{M}}=\left\langle 1, X_{2}^{2} \right\rangle \nonumber
\end{equation}
we have that:
\begin{equation}
U_{\textbf{M}}\cap\R[X_{1},X_{2}]_{1}=\left\langle-\alpha+X_{1},-\beta+X_{2}\right\rangle\nonumber
\end{equation}
We need to add the polynomial $1$ to $U_{\textbf{M}}\cap\R[X_{1},X_{2}]_{1}$ to get basis of $\R[X_{1},X_{2}]_{1}$ therefore we have that:
\begin{equation}
 \frac{\R[X_{1},X_{2}]_{1}}{U_{\textbf{M}}\cap\R[X_{1},X_{2}]_{1}}=\left\langle \overline{1}^{\textbf{M}'} \right\rangle\nonumber
\end{equation}
Thence by Remark \ref{auxiliar1} we get that: 
\begin{equation}
T_{\textbf{M}}=\left\langle \overline{1}^{L} \right\rangle\nonumber
\end{equation}

Since  $v:=\{ \overline{1}^{\textbf{M}}\}$ is also an orthonormal basis with respect to the GNS product of $L$ we can directly compute the matrices of truncated GNS multiplication operators of $\textbf{M}$:

\begin{align*}
M(M_{\textbf{M},X_{1}},v)=\poly^{-1}(X_{1}1)\textbf{M}\poly^{-1}(1)= &\left(\begin{array}{llllll}
0 & 1 & 0 & 0 & 0 & 0\\
  \end{array}\right)\textbf{M} \left(\begin{array}{l}
1\\
0\\
0\\
0\\
0\\
0\\ 
  \end{array}\right)=\left(\alpha \right)
\\
M(M_{\textbf{M},X_{2}},v)=\poly^{-1}(X_{2}1)\textbf{M}\poly^{-1}(1)= &\left(\begin{array}{llllll}
0 & 0 & 1 & 0 & 0 & 0\\
  \end{array}\right)\textbf{M} \left(\begin{array}{l}
1\\
0\\
0\\
0\\
0\\
0\\ 
  \end{array}\right)=\left(\beta\right)
\end{align*}

Therefore:
\begin{equation}
\widetilde{\textbf{M}}=V_{2}(\alpha,\beta)V_{2}(\alpha,\beta)^{T}\nonumber
\end{equation}
Then \textbf{M} admits a Gaussian quadrature rule. However it does not admit a quadrature rule. Indeed, suppose $\textbf{M}$ admits a quadrature rule  with $N$ nodes, then according to \ref{extremalbound}:
\begin{equation}
 2=\rk\textbf{M}\leq N \leq |V_{\C}(U_{\textbf{M}})| \nonumber
\end{equation}
But can easily see that $V_{\C}(U_{\textbf{M}})=\left(\alpha,\beta\right)$ and
\begin{equation} 
\rk \textbf{M}=2>|V_{\C}(\textbf{U}_{\textbf{M}})|=1 \nonumber
\end{equation}
prevents to  $\textbf{M}$ to have a quadrature rule.
\end{ej}
\begin{ej}
Let us consider the following generalized Hankel matrix in two variables of order $2$ taken from \cite[Example 1.13]{cf3}:
\begin{equation}
M = \begin{blockarray}{ccccccc}
& 1 & X_{1} & X_{2} & X_{1}^{2} & X_{1}X_{2} & X_{2}^{2}\\
\begin{block}{c(ccc|ccc)}
 1 &   1  &  1 &  2 &  2 &  0  &  3 \\
 X_{1} &   1  & 2  &  0 &  4   &  0  &  0  \\
	X_{2} &	1 & 0 & 3 & 0 & 0 & 9 \\ \cline{2-7}
 X_{1}^{2} &  2 &  4 &  0 &   9 &  0  &  0 \\
 X_{1}X_{2} &   0 &   0 &   0 &   0  &  0 &   0  \\
 X_{2}^{2} &  3  &  0  &  9 &   0  &  0 &  28  \\
	\end{block}
	\end{blockarray}
   \nonumber
\end{equation}
This matrix does not have a quadrature rule representation with the minimal number of nodes, as has been proved in \cite[Example 1.13]{cf3}, however it admits a Gaussian quadrature rule. Indeed, we can compute with the truncated GNS construction that: 
\begin{equation}
\widetilde{M}=\frac{1}{6}V_{2}(0,0)V_{2}(0,0)^{T}+\frac{1}{3}V_{2}(0,3)V_{2}(0,3)^{T}+\frac{1}{2}V_{2}(2,0)V_{2}(2,0)^{T}. \nonumber
\end{equation}
\end{ej}

The following corollary is a very well known result of Curto and Fialkow (see \cite[corollary 5.14]{cf0} ) in terms of quadrature rules instead of nonnegative measures. In \cite{byt} there is a proof about the correspondence between quadrature rules and nonnegative measures. This result of Curto and Fialkow uses tools of functional analysis like the the Spectral theorem and the Riesz representation theorem. Monique Laurent gave also a more elementary proof (see \cite[corollary 1.4]{rev} ) that uses a corollary of the  Hilbert Nullstellensatz and elementary linear algebra. The main contribution of this proof is that it does not need to find a flat extension of the linear form since the truncated GNS multiplication operators commute and we can apply directly the Theorem \ref{st}, and despite of it uses the Hilbert Nullstellensatz in the proof of Theorem \ref{lio}, we do not need to apply the  Hilbert Nullstellensatz to show that the nodes are in $\R^{n}$, since the nodes are real because its coordinates are the eigenvalues of a real symmetric matrix.

\begin{cor}\label{implicacion}
Suppose $L$ is flat then $L$ has a quadrature rule with $\rk(M_{L})$ many nodes (the minimal number of nodes).
\end{cor}
\begin{proof}
If $L$ is flat by Proposition \ref{importante2} the  truncated GNS multiplication operators of $L$ commute and applying \ref{importante} then $L$ has a quadrature rule on $G_{L}$  \eqref{semiraro}, with $\dim(T_{L})\stackrel{L \text{ is flat }}{=} \dim(V_{L})=\rk(M_{L})$ many nodes. Since $L$ is flat $\R[\underline{X}]_{d+1}=\R[\underline{X}]_{d}+U_{L}$ and therefore one can easily see that $G_{L}=\R[\underline{X}]_{2d+2}$. As a conclusion we get a quadrature rule for $L$ with $\rk(M_{L})$ many nodes.
\end{proof}

\section{Main Theorem}

In this section we will demonstrate that the commutativity of the truncated GNS multiplication operators of $L$ is equivalent to the matrix $W^{T}_{L}A_{L}W_{L}$ being Hankel.

 \begin{main}\label{main}

The following assertions are equivalent:
\begin{enumerate}
\item The truncated multiplication operators $M_{L,1},\ldots, M_{L,n}$ pairwise commute.
\item There exists $\hat{L}\in\R[\underline{X}]_{2d+2}^{*}$ such that $L=\hat{L}$ on $\R[\underline{X}]_{2d+1}$ and $\hat{L}$ is flat.
\end{enumerate}
\end{main}

\begin{proof}
1$\Rightarrow $2. By the theorem \ref{importante} there exist $a_1,\ldots,a_N\in\R^{n}$ pairwise different nodes and $\lambda_{1}>0,\ldots,\lambda_{N}>0$ weights, where $N:=\dim(T_{L})$ such that: $L(p)=\sum_{i=1}^{N}\lambda_{i}p(a_{i})$ for all $p\in G_{L}$, where $G_{L}$ was defined in \eqref{semiraro}. Let us define, $\hat{L}:=\sum_{i=1}^{N}\lambda_{i}\ev_{a_{i}}\in\R[\underline{X}]_{2d+2}^{*}$. We have shown in theorem \ref{importante} that $\hat{L}=L$ on $\R[\underline{X}]_{2d+1}$, $U_{L}\subset U_{\hat{L}}$ and obviously $\tilde{L}(\sum\R[\underline{X}]_{d+1}^{2})\subseteq \R_{\geq 0}$, so it remains to show that $\tilde{L}$ is flat, that is to say:
\begin{equation}
\dim V_{\hat{L}}=\dim T_{\hat{L}}\nonumber
\end{equation}
or equivalently using \ref{dimension}, it remains to show:
\begin{equation}
 \dim(\frac{\R[\underline{X}]_{d+1}}{U_{\hat{L}}})=\dim(\frac{\R[\underline{X}]_{d}}{U_{\hat{L}}\cap\R[\underline{X}]_{d}})\nonumber
 \end{equation}
Since $U_{\hat{L}}\cap\R[\underline{X}]_{d}=U_{L}\cap\R[\underline{X}]_{d}$ and using again proposition \ref{dimension}, we have the following:
\begin{equation}
\dim(\frac{\R[\underline{X}]_{d}}{U_{\hat{L}}\cap\R[\underline{X}]_{d}})=\dim(\frac{\R[\underline{X}]_{d}}{U_{L}\cap\R[\underline{X}]_{d}})\stackrel{\ref{dimension}}{=}\dim(T_{L})=N \nonumber
\end{equation}

therefore, in the following we will prove $\dim(\frac{\R[\underline{X}]_{d+1}}{U_{\tilde{L}}})=N$. For this, let us consider the following linear map, between euclidean vector spaces:

\begin{equation}\label{canonicalmap2}
\frac{\R[\underline{X}]_{d+1}}{U_{\hat{L}}} \hookrightarrow \frac{\R[\underline{X}]}{U_{\Lambda}},\overline{p}^{\hat{L}}\mapsto\overline{p}^{\Lambda}
\end{equation}
where $\Lambda:=\sum_{i=1}^{N}\lambda_{i}\ev_ {a_{i}}\in\R[\underline{X}]^{*}$. Notice that the canonical map \eqref{canonicalmap2} is well defined since $U_{\hat{L}}= U_{\Lambda}\cap\R[\underline{X}]_{d}$ and therefore it is injective. Then 

\begin{equation}
 \dim(\frac{\R[\underline{X}]_{d+1}}{U_{\hat{L}}})\leq \dim(\frac{\R[\underline{X}]}{U_{\Lambda}})\stackrel{\ref{zero}}{=} N \nonumber
\end{equation}

It remains to show $N\leq \dim(\frac{\R[\underline{X}]_{d+1}}{U_{\hat{L}}})$. But this is true, since:
\begin{equation}
 N=\dim(T_{L})\stackrel{\ref{dimension}}{=}\dim(\frac{\R[\underline{X}]_{d}}{U_{L}\cap\R[\underline{X}]_{d}})= \dim(\frac{\R[\underline{X}]_{d}}{U_{\hat{L}}\cap\R[\underline{X}]_{d}})\leq \dim (\frac{\R[\underline{X}]_{d+1}}{U_{\hat{L}}})\nonumber
\end{equation}

2$\Rightarrow$1. Since $\hat{L}$ is flat, then by \ref{importante2} we know that the  truncated GNS multiplication operators of $\hat{L}$ pairwise commute. Then by applying again \ref{importante} there exists $a_{1},\ldots,a_{N}\in\R^{n}$, pairwise different nodes, and $\lambda_{1}>0,\ldots,\lambda_{N}>0$ weights, with $N=\dim( T_{\hat{L}})$ such that if we set $\Lambda:=\sum_{i=1}^{N}\lambda_{i}\ev_{a_{i}}\in\R[\underline{X}]^{*}$, we get $\Lambda(p)=\hat{L}(p)=L(p)$ for all $p\in\R[\underline{X}]_{2d+1}$, and $U_{L}\subseteq U_{\hat{L}}\subseteq U_{\Lambda}$. Indeed notice that $U_{L}\subseteq U_{\hat{L}}$ since for $p\in U_{L}$, $L(pq)=\hat{L}(pq)=0$ for all $q\in\R[\underline{X}]_{d}$, and since $\hat{L}$ is flat this implies $p\in U_{ \hat{L}}$. Obviously $M_{\Lambda,i}$ pairwise commute for all $i\in\{1,\ldots,n\}$, since they are the original GNS operators modulo $U_{\Lambda}$ defined in \ref{GNSmodule}. In order to prove that $M_{L,i}$ pairwise commute for all $i\in\{1,\ldots,n\}$, let us first consider the linear isometry $\sigma_{1}$ \eqref{sigma1} of the proposition \ref{fbound}. Since $\sigma_{1}$ is a linear isometry is inmmediately injective, and then $\dim (T_{L})\leq \dim (\frac{\R[\underline{X}]}{U_{\Lambda}})$. Therefore we have the following inequalities:

\begin{equation*}
\begin{aligned}
N=\dim(T_{\hat{L}})=\dim(\frac{\R[\underline{X}]_{d}}{U_{\tilde{L}}\cap\R[\underline{X}]_{d}})
=\dim(\frac{\R[\underline{X}]_{d}}{U_{L}\cap\R[\underline{X}]_{d}})= &\\ \dim (T_{L})\leq \dim (\frac{\R[\underline{X}]}{U_{\Lambda}})\stackrel{\ref{zero}}{=}N &\\
\end{aligned}
\end{equation*}

then $\dim(T_{L})=\dim (\frac{\R[\underline{X}]}{U_{\Lambda}})$. Then $\sigma_{1}$, in this case, is in particular surjective and in conclusion is an isomorphism. With this result we be able to prove that the following diagram is commutative, for all $i\in\{1,\ldots,n\}$: 

\begin{equation}\label{diagram}
\xymatrix{
T_{L} \ar[d]^{\sigma_{1}} \ar[r]^{M_{L,i}} & T_{L} \\
\frac{\R[\underline{X}]}{U_{\Lambda}} \ar[r]^{M_{\Lambda,i}} &\frac{\R[\underline{X}]}{U_{\Lambda}} \ar[u]_{\sigma^{-1}_{1}} }
\end{equation}

That is to say $M_{L,i}=\sigma_{1}^{-1}\circ M_{\Lambda,i}\circ \sigma_{1}$. To show this let $p,q\in\R[\underline{X}]_{d}$, then we have:

\begin{center}
$\left\langle M_{L,i}(\overline{p}^{L}),\overline{q}^{L} \right\rangle_{L} = \left\langle \Pi_{L}(\overline{X_{i}p}^{L}),\overline{q}^{L} \right\rangle_{L}\stackrel{\Pi_{L}\circ\Pi_{L}=\Pi_{L}}{=}\left\langle\overline{X_{i}p}^{L},\overline{q}^{L} \right\rangle_{L}
=L(X_{i}pq)\stackrel{\Lambda=L\text{ on }\R[\underline{X}]_{2d+1}}{=}\Lambda(X_{i}pq)=\left\langle \overline{X_{i}p}^{\Lambda},\overline{q}^{\Lambda}\right\rangle_{\Lambda}=\left\langle \sigma_{1}\circ\sigma^{-1}_{1}(\overline{X_{i}p}^{\Lambda}),\overline{q}^{\Lambda} \right\rangle_{\Lambda} = \left\langle \sigma^{-1}_{1}(\overline{X_{i}p}^{\Lambda}),\sigma^{-1}_{1}(\overline{q}^{\Lambda}) \right\rangle_{L} =\left\langle \sigma^{-1}_{1}\circ M_{\Lambda,i}(\overline{p}^{\Lambda}),\overline{q}^{L}\right\rangle_{L}= \left\langle \sigma^{-1}_{1}\circ M_{\Lambda,i}\circ\sigma_{1} (\overline{p}^{L}),\overline{q}^{L} \right\rangle_{L}$
\end{center}

Finally we can conclude that the truncated GNS multiplication operators of $L$ pairwise commute, using the commutativity of the GNS multiplication operators of $\Lambda$. Indeed:
 
\begin{align}\label{aria}
M_{L,i}\circ M_{L,j}=\sigma^{-1}_{1}\circ M_{\Lambda,i}\circ\sigma_{1}\circ\sigma^{-1}_{1}\circ M_{\Lambda,j}\circ\sigma_{1}=\sigma^{-1}_{1}\circ M_{\Lambda,i}\circ M_{\Lambda,j}\circ\sigma_{1}=& \\ \nonumber
\sigma^{-1}_{1}\circ M_{\Lambda,j}\circ M_{\Lambda,i}\circ\sigma_{1}=\sigma^{-1}_{1}\circ M_{\Lambda,j}\circ\sigma_{1}\circ\sigma^{-1}_{1}\circ M_{\Lambda,i}\circ\sigma_{1}=M_{L,j}\circ M_{L,i}&
\end{align}
\end{proof}
\begin{teor}\label{matrices}
The following assertions are equivalent:
\begin{enumerate}
\item $M_{L,1},\ldots,M_{L,n}$ pairwise commute. 
\item $\widetilde{M_{L}}$ is a Generalized Hankel matrix
\end{enumerate}
\end{teor}
\begin{proof}
$(1)\Rightarrow(2)$ Assume $M_{L,1},\ldots,M_{L,n}$ pairwise commute. Then by Theorem \ref{main} there exists a linear form $\hat{L}\in\R[\underline{X}]^{*}_{2d+2}$ such that: $\hat{L}(\R[\underline{X}]^{2}_{d+1})\subseteq \R_{\geq 0}$, $L=\hat{L}$ on $\R[\underline{X}]_{2d+1}$ and $\hat{L}$ is flat, what implies by \ref{particular} that $\hat{L}$ has a quadrature rule representation and therefore $M_{\hat{L}}\succeq 0$ and $M_{\hat{L}}$ is a Generalized Hankel matrix. It is enough to show that $C_{\hat{L}}=W^{T}_{L}A_{L}W_{L}$. The last follows from the fact that $\hat{L}$ is flat and $L=\hat{L}$ on $\R[\underline{X}]_{2d+1}$ since:
\begin{center}
 $$\rk(M_{\hat{L}})=\rk
\left(
\begin{array}{c|c}
\makebox{$A_{L}$}&\makebox{$A_{L}W_{L}$}\\
\hline
  \vphantom{\usebox{0}}\makebox{$W_{L}^{T}A_{L}$}&\makebox{$C_{\hat{L}}$}
\end{array}
\right)\stackrel{\hat{L}\text{ is flat }}{=}\rk(A)\iff C_{\hat{L}}=W_{L}^{T}A_{L}W_{L}
$$ 
\end{center}
$(2)\Rightarrow(1)$ Suppose $\widetilde{M_{L}}$ is a Generalized Hankel matrix, and denote $\hat{L}:=L_{\widetilde{M_{L}}}\in\R[\underline{X}]^{*}_{2d+2}$. Since $\widetilde{M_{L}}$ is the moment matrix of the linear form $\hat{L}\in\R[\underline{X}]^{*}_{2d+2}$ then $\hat{L}$ is flat and by Theorem \ref{importante2} the truncated GNS multiplication operators of $\hat{L}$ commute. Now, to prove the truncated GNS operators commute let us define $\sigma:=\sigma_{\hat{L}}^{-1}\circ\sigma_{L}$ an isomorphism of euclidean vector spaces. We will prove that the following diagram is commutative:

\begin{equation}\label{diagramequis}
\xymatrix{
T_{L} \ar[d]^{\sigma} \ar[r]^{M_{L,i}} & T_{L} \\
T_{\hat{L}} \ar[r]^{M_{\hat{L},i}} & T_{\hat{L}} \ar[u]_{\sigma^{-1}} }\nonumber
\end{equation}

The diagram is commutative if and only if $M_{L,i}=\sigma\circ M_{\hat{L},i}\circ\sigma^{-1}$. To prove this equality let us take $p,q\in\R[\underline{X}]_{d}$. Then:

\begin{center}
$\left\langle M_{L,i}(\overline{p}^{L}),\overline{q}^{L} \right\rangle_{L} = \left\langle \Pi_{L}(\overline{X_{i}p}^{L}),\overline{q}^{L} \right\rangle_{L}\stackrel{\Pi_{L}\circ\Pi_{L}=\Pi_{L}}{=}\left\langle\overline{X_{i}p}^{L},\overline{q}^{L} \right\rangle_{L}
=L(X_{i}pq)\stackrel{\hat{L}=L\text{ on }\R[\underline{X}]_{2d+1}}{=}\hat{L}(X_{i}pq)=\left\langle \overline{X_{i}p}^{\hat{L}},\overline{q}^{\hat{L}}\right\rangle_{\hat{L}}=\left\langle \overline{X_{i}p}^{\hat{L}},\Pi_{\hat{L}}(\overline{q}^{\hat{L}}) \right\rangle_{\hat{L}}=\left\langle \Pi_{\hat{L}}(\overline{X_{i}p}^{\hat{L}}),\overline{q}^{\hat{L}}, \right\rangle_{\hat{L}}=\left\langle \sigma\circ\sigma^{-1}(\Pi_{\hat{L}}(\overline{X_{i}p}^{\hat{L}})),\overline{q}^{\hat{L}} \right\rangle_{\hat{L}} = \left\langle \sigma^{-1}(\Pi_{\hat{L}}(\overline{X_{i}p}^{\hat{L}})),\sigma^{-1}(\overline{q}^{\hat{L}}) \right\rangle_{L} =\left\langle \sigma^{-1}\circ M_{\hat{L},i}(\overline{p}^{\hat{L}}),\overline{q}^{L}\right\rangle_{L}= \left\langle \sigma^{-1}\circ M_{\hat{L},i}\circ\sigma (\overline{p}^{L}),\overline{q}^{L} \right\rangle_{L}$
\end{center}

Finally we can conclude the truncated GNS multiplication operators of $L$ commute using the commutativity of the truncated GNS multiplication operators of $\hat{L}$ in the identical way we already did in the previous Theorem in \eqref{aria}.
\end{proof}

\begin{cor}\label{otroflat}
Suppose $L$ is flat, then $\widetilde{M}$ is a generalized Hankel matrix.
\end{cor}
\begin{proof}
If $L$ is flat then by Theorem \ref{importante2}  the truncated GNS multiplication operators of $L$ commute, and therefore by Theorem \ref{matrices} we get that this is equivalent to $\widetilde{M}_{L}$ being a generalized Hankel matrix.
\end{proof}

The following result uses the Theorem \ref{main} together with ideas from \cite{Put} and give us a generalization of a classical Theorem from  Mysovskikh \cite{mysov}, Dunkl and Xu \cite[Theorem 3.8.7]{xu} and Putinar  \cite[pages 189-190]{Put}. They proved the equivalence between the existence of a minimal Gaussian quadrature rule  with the commutativity of the truncated  GNS  multiplication operators for a positive definite linear form on $\R[\underline{X}]$. The generalization here comes from the fact that the result holds also if the linear form is defined on $\R[\underline{X}]_{2d+2}$ for $d\in\N_{0}$ and it is positive semidefinite   i.e. we do not assume $U_{L}=\{0\}$. We also provide a third equivalent condition in the result which is $W^{T}_{L}A_{L}W_{L}$ is a generalized Hankel matrix, a fact which seems no to have been noticed so far.

\begin{cor}\label{xu}
The following assertions are equivalent:
\begin{enumerate}
\item The linear form $L$ admits a Gaussian quadrature rule.
\item The truncated GNS multiplication operators of $L$ commute.
\item $\widetilde{M_{L}}$ is a generalized Hankel matrix.
\end{enumerate}
\end{cor}

\begin{proof}
$(1)\Rightarrow (2)$. Assume that $L$ admits a Gaussian quadrature rule, that is to say $L(p)=\sum_{i=1}^{N}\lambda_{i}p(a_{i})$ for all $p\in\R[\underline{X}]_{2d+1}$ where $N:=\dim(T_{L})$, the points $a_{1},\ldots,a_{N}$ are pairwise different and $\lambda_{1}>0,\ldots,\lambda_{N}>0$. Let us set $\Lambda:=\sum_{i=1}^{N}\lambda_{i}\ev_{a_{i}}\in\R[\underline{X}]^{*}$. Using \ref{zero} we have the following:
\begin{equation}\label{susto}
\dim(\frac{\R[\underline{X}]}{U_{\Lambda}})=N=\dim T_{L}
\end{equation}
Let us consider again the linear isometry $\sigma_{1}$, already defined in \eqref{sigma1}:
\begin{align}
\sigma_{1}:T_{L}\longrightarrow \frac{\R[\underline{X}]}{U_{\Lambda}}, \overline{p}^{L}\mapsto\overline{p}^{\Lambda}\nonumber
\end{align} 
As we proved in \ref{fbound} is well defined and is an isometry, what implies $\sigma_{1}$ is injective, and considering that in this case  it holds \eqref{susto}, $\sigma_{1}$ is moreover an isomorphism. We continue as in the implication 2$\Rightarrow$1 of the proof of Theorem \ref{main}, showing that the diagram \eqref{diagram} is commutative, what together with the fact that $M_{\Lambda,i}$ always commute for all $i \in\{1,\ldots,n\}$ implies that the truncated GNS multiplication operators of $L$ commute. \\
$(2)\Rightarrow(1)$. This part was alredy proved in the Remark \ref{particular} as a consequece of the Theorem \ref{importante}.\\
$(2)\iff (3)$ It is the Theorem \ref{matrices}.
\end{proof}

The following result of M\"oller will give us a better lower bound in the number of nodes of a quadrature rule on $\R[\underline{X}]_{2d+1}$ than the very well-known bound given in Proposition \ref{fbound}. This bound, was already found for positive linear forms by M\"oller in 1975 and by Putinar in 1997 (\cite{moller},\cite{Put}). This result will show that the bound it is also true for positive semidefinite linear forms and it uses the same ideas as in \cite{Put}. We include the proof for the convenience of the reader. This bound will help us in polynomial optimization problems in which we know the number of global minimizers in advance, to discard optimality if this bound is bigger than the number of global minimizers, see Example \ref{madrugada} below.

\begin{teor}\label{moller}
The number of nodes $N$ of a Gaussian quadrature rule for $L$ satisfies:
\begin{equation}\label{cyber}
N\geq \dim (T_{L})+\frac{1}{2}\max_{1\leq j,k\leq n}(\rk[M_{L,j},M_{L,k}])
\end{equation}
\end{teor}
\begin{proof}
Assume $L$ has a  quadrature rule with $N$ nodes, that is to say, there exist  $\lambda_{1}>0,\ldots,\lambda_{N}>0$ weights and  $a_{1},\ldots,a_{N}$ in $\R^{n}$ pairwise different nodes, such that $L(p)=\sum_{i=1}^{N}\lambda_{i}p(a_{i})$ for all $p\in\R[\underline{X}]_{2d+1}$. Let us set $\Lambda:=\sum_{i=1}^{N}\lambda_{i}\ev_{a_{i}}\in\R[\underline{X}]^{*}$. By using the proposition \ref{zero} we have that:

\begin{center}
 $\dim(\frac{\R[\underline{X}]}{U_{\Lambda}})=N<\infty$. 
\end{center}
Then we can choose an orthonormal basis of $\frac{\R[\underline{X}]}{U_{\Lambda}}$. Let us denote such a basis by $\beta_{\Lambda}:=\{\overline{\beta_{1}}^{\Lambda}, \ldots ,\overline{\beta_{N}}^{\Lambda} \}$ for $\beta_{1},\ldots,\beta_{N}\in\R[\underline{X}]$ pairwise different. Then we have that the transformation matrix of the multiplication operators $M_{\Lambda,i}$ with respect to this orthonormal basis is:
\begin{equation}
(\Lambda(X_{i}\beta_{k}\beta_{j}))_{1\leq k,j \leq N}\nonumber
\end{equation}
The set $A:=\{\text{ }\overline{p}^{\Lambda}\text{ }|\text{ }p\in\R[\underline{X}]_{d}\}$ is a subspace of $\frac{\R[\underline{X}]}{U_{\Lambda}}$ so we can assume without loss of generality that $\beta_{1},\ldots,\beta_{r}\in\R[\underline{X}]_{d}$ where $r:=\dim A$ generate a basis of A. Then since $L=\Lambda$ on $\R[\underline{X}]_{2d+1}$, we obtain:
\begin{equation}\label{transformation}
 (M(M_{\Lambda,i},\beta_{\Lambda})):=(\Lambda(X_{i}\beta_{k}\beta_{j}))_{1\leq k,j \leq N}=\left(
\begin{array}{c|ccc}
 (L(X_{i}\beta_{k}\beta_{j}))_{1\leq k,j \leq r} & B_{i}\\
\hline
B_{i}^{t}& C_{i}\\
\end{array}
\right)
\end{equation}
$(M(M_{\Lambda,i},\beta_{\Lambda}))$ is the transformation matrix of te $i$-th truncated GNS multiplication operator of $\Lambda$ with respecto to the basis $\beta_{\Lambda}$ and where $B_{i}\in\R^{r\times N-r}$ and $C_{i}\in\R^{N-r\times N-r}$ are symmetric matrices. We will show that $\beta_{L}:=\{\overline{\beta_{1}}^{L}, \ldots, \overline{\beta_{r}}^{L}\}$ is an orthonormal basis of $T_{L}$.

Taking $\sigma_{1}$  the isometry defined in \eqref{sigma1}, we get:
\begin{equation}
\sigma_{1}(T_{L})=\{ \sigma_{1}(\overline{p}^{L})\text{ }|\text{ }\overline{p}^{L}\in T_{L} \}=\{\sigma_{1}(\overline{p}^{L})\text{ }|\text{ }p\in\R[\underline{X}]_{d}\}=\{\text{ }\overline{p}^{\Lambda}\text{ }|\text{ }p\in\R[\underline{X}]_{d}\}=A\nonumber
\end{equation}
Hence, we get:
\begin{equation}
 \sigma_{1}(T_{L})=\langle \overline{\beta_{1}}^{\Lambda},\ldots,\overline{\beta_{r}}^{\Lambda}\rangle \nonumber
\end{equation}
And since we have chosen $\beta_{i}\in\R[\underline{X}]_{d}$ for all $i\in \{1,\ldots,r\}$, then we have $\sigma_{1}(\overline{\beta_{i}}^{L})=\overline{\beta_{i}}^{\Lambda}$. Therefore $\beta_{L}:=\{\overline{\beta_{1}}^{L},\ldots,\overline{\beta_{r}}^{L}\}$ generate a basis of $T_{L}$. It remains to show that $\beta_{L}$ is orthonormal. To see that $\beta_{L}$ is orthonormal we use again the fact that $\sigma_{1}$ is an isometry and that $\sigma_{1}(T_{L})=A$. Indeed for $1\leq i,j \leq r$:

\begin{center}
$\delta_{ij}=\Lambda(\beta_{i}\beta_{j})=\langle \overline{\beta_{j}}^{\Lambda},\overline{\beta_{i}}^{\Lambda} \rangle_{\Lambda}\stackrel{}{=}\langle \sigma^{-1}_{1}(\overline{\beta_{j}}^{\Lambda}),\sigma^{-1}_{1}(\overline{\beta_{i}}^{\Lambda})\rangle_{L}=\langle \overline{\beta_{j}}^{L},\overline{\beta_{i}}^{L}\rangle_{L}$
\end{center}
Therefore we have shown that:
$$M(M_{\Lambda,i},\beta_{\Lambda})=
\left(
\begin{array}{c|c}
\makebox{$M(M_{L,i},\beta_{L})$}&\makebox{$B_{i}$}\\
\hline
  \vphantom{\usebox{0}}\makebox{$B_{i}^{T}$}&\makebox{$C_{i}$}
\end{array}
\right)
$$
where  we use the notation:
\begin{align}
(M(M_{L,i},\beta_{L})):=(L(X_{i}\beta_{k}\beta_{j}))_{1\leq k,j \leq N} \nonumber
\end{align}
to refer us to the transformation matrix of the $i$-th truncated GNS multiplication operators of $L$ with respect to the basis $\beta_{L}$. Using the fact that the matrices $M(M_{\Lambda,i},\beta_{\Lambda})$ commute, we have the following equality:
\begin{equation}
M(M_{L,j},\beta_{L})M(M_{L,i},\beta_{L})-M(M_{L,i},\beta_{L})M(M_{L,j},\beta_{L})=B_{i}B_{j}^{T}-B_{j}B_{i}^{T}\nonumber
\end{equation}
Therefore the following it holds:
\begin{equation}
\rk(B_{i}B_{j}^{T}-B_{j}B_{i}^{T})\leq 2\rk(B_{i}B_{j}^{T})\leq 2\rk(B_{i})\leq 2\min \{r,N-r\} \leq 2 (N-r)\nonumber
\end{equation}
and then:
\begin{center}
 $\rk[M(A_{L,j},\beta_{L}),M(A_{L,i},\beta_{L})]\leq 2(N-r)$
\end{center}
Since we have already proved $r=\dim T_{L}$. And then we can conclude:
\begin{equation}
N\geq \dim (T_{L})+\frac{1}{2}\max_{1\leq j,k\leq n}(\rk[M_{L,j},M_{L,k}])\nonumber
\end{equation}
\end{proof}
\begin{rem}
Note that we can use the previous Theorem \ref{moller} to show in a different way $(1)\Rightarrow(2)$ in the Corollary \ref{xu}. Indeed, let us suppose that $L$ has a Gaussian quadrature rule that is to say with $N=\dim T_{L}$ nodes. Using the inequality \eqref{cyber} we get that $\rk[M_{L,j},M_{L,k}]=0$ for $j,k\in\{ 1,\ldots,n\}$ therefore the truncated GNS multiplication operators of $L$ commute.
\end{rem}

\begin{ej}\label{madrugada}
Let us consider the following polynomial optimization problem taken from \cite{las}:
\begin{equation*}
\begin{aligned}
& {\text{minimize}}
& & f(x)=x_{1}^2x_{2}^2(x^2+y^2-1) \\
& \text{subject to}
& & x_{1},x_{2}\in\R
\end{aligned}
\end{equation*}
\end{ej}
By Calculus we know that the minimizers of $f$ occur in the real points common to the partial derivatives of $f$ (the real gradient variety) and we can easily check that this derivatives intersect in $4$ real points: $\left( \pm\frac{1}{\sqrt{3}},\pm\frac{1}{\sqrt{3}} \right)\in\R^{2}$. Therefore we know in advance that $(P)$ has at most $4$ minimizers. On other side, an optimal solution of the moment relaxation of order $8$ $(P_{8})$, that is $\textbf{M}:=M_{8,1}(y)$ read as:

{\scriptsize
\begin{align}
\left(\begin{array}{rrrrrrrrrrrrrrr} \nonumber
1.00&0.00&0.00&62.12&-0.00&62.12&0.00&0.00&0.00&0.00\\
0.00&62.12&-0.00&0.00&0.00&0.00&9666.23&-0.00&8.33&-0.00\\
0.00&-0.00&62.12&0.00&0.00&0.00&-0.00&8.33&-0.00&9666.23\\
62.12&0.00&0.00&9666.23&-0.00&8.33&0.00&-0.00&0.00&0.00\\
-0.00&0.00&0.00&-0.00&8.33&-0.00&-0.00&0.00&0.00&-0.00\\
62.12&0.00&0.00&8.33&-0.00&9666.23&0.00&0.00&-0.00&0.00\\
0.00&9666.23&-0.00&0.00&-0.00&0.00&3150633.17&-0.00&2.27&0.00\\
0.00&-0.00&8.33&-0.00&0.00&0.00&-0.00&2.27&0.00&2.27\\
0.00&8.33&-0.00&0.00&0.00&-0.00&2.27&0.00&2.27&-0.00\\
0.00&-0.00&9666.23&0.00&-0.00&0.00&0.00&2.27&-0.00&3150630.69\\
9666.23&0.00&-0.00&3150633.17&-0.00&2.27&0.42&-0.00&-0.00&0.00\\
-0.00&-0.00&0.00&-0.00&2.27&0.00&-0.00&-0.00&0.00&0.00\\
8.33&0.00&0.00&2.27&0.00&2.27&-0.00&0.00&0.00&-0.00\\
-0.00&0.00&-0.00&0.00&2.27&-0.00&0.00&0.00&-0.00&-0.00\\
9666.23&-0.00&0.00&2.27&-0.00&3150630.69&0.00&-0.00&-0.00&0.33
\end{array}\right.\\ \nonumber
\left.\begin{array}{rrrrrrrrrrrrrrr}
9666.23&-0.00&8.33&-0.00&9666.23\\
0.00&-0.00&0.00&0.00&-0.00\\
-0.00&0.00&0.00&-0.00&0.00\\
3150633.17&-0.00&2.27&0.00&2.27\\
-0.00&2.27&0.00&2.27&-0.00\\
2.27&0.00&2.27&-0.00&3150630.69\\
0.42&-0.00&-0.00&0.00&0.00\\
-0.00&-0.00&0.00&0.00&-0.00\\
-0.00&0.00&0.00&-0.00&-0.00\\
0.00&0.00&-0.00&-0.00&0.33\\
2466755083.36&-43.48&169698627.89&-6.08&134568970.57\\
-43.48&169698627.89&-6.08&134568970.57&15.08\\
169698627.89&-6.08&134568970.57&15.08&169698562.66\\
-6.08&134568970.57&15.08&169698562.66&25.61\\
134568970.57&15.08&169698562.66&25.61&2466752654.76
\end{array} \right)   \nonumber
\end{align}}
and the rank of the commutator of the truncated GNS multiplication operators is:
{\scriptsize \begin{align}
&\rk[M_{\textbf{M},X_{1}},M_{\textbf{M},X_{1}}]\nonumber= \\ \nonumber
&\rk \left(\begin{array}{rrrrrrrrrr} \nonumber
             0     &     0.00    &      0.00  &       -0.00   &      -0.00   &       0.00    &      0.00  &       -0.00   &       0.00     &    -0.00 \\
         -0.00    &         0   &      -0.00    &      0.00   &       0.00  &        0.00     &    -0.00  &        0.00   &       0.00   &      -0.00 \\
         -0.00     &     0.00    &         0     &     0.00     &     0.00    &      0.00    &      0.00 &        -0.00   &       0.00   &       0.00 \\
          0.00   &      -0.00   &      -0.00     &        0    &     -0.00     &    -0.00    &     -0.00 &         -0.00  &        0.00   &      -0.00 \\
          0.00  &       -0.00  &       -0.00     &     0.00    &         0    &      0.00    &      0.00 &          0.00   &       0.00  &        0.00 \\
         -0.00  &       -0.00  &       -0.00   &       0.00     &    -0.00    &         0    &     -0.00 &        -0.00   &       0.00   &      -0.00 \\
         -0.00  &        0.00  &       -0.00  &        0.00    &     -0.00    &      0.00    &         0 &      -313.91    &     -0.00   &     115.50 \\
          0.00    &     -0.00   &       0.00   &       0.00    &     -0.00    &      0.00    &    313.91 &            0     &     0.18    &      0.00 \\
         -0.00    &     -0.00   &      -0.00   &      -0.00    &     -0.00    &     -0.00   &       0.00 &        -0.18     &        0    &      0.05 \\
          0.00    &       0.00  &       -0.00  &        0.00    &     -0.00   &       0.00  &     -115.50&         -0.00    &     -0.05   &          0 
\end{array}\right)=4\nonumber
\end{align}}
If $\textbf{M}$ had a quadrature rule on $\R[X_{1},X_{2}]_{7}$ with $N$ nodes, since $f\in\R[X_{1},X_{2}]_{7}$ by \ref{motiv} $(iii)$ the $N$ nodes of the quadrature rule would be global minimizers of $f$ and $P^{*}=P^{*}_{8}$, and according to Theorem \ref{moller}:
\begin{equation}
N\geq \dim (T_{L})+\frac{1}{2}\max_{1\leq j,k\leq n}(\rk[M_{L,j},M_{L,k}])=10 + \frac{1}{2}4=12 \nonumber
\end{equation}
Therefore the polynomial $f$ would have at least $12$ global minimizers and this is a contradiction with the fact that $f$ has at most $4$ global minimizers. Notice that then $\textbf{M}$ does not have a  quadrature rule on $\R[\underline{X}]_{7}$, and in particular it does not have a quadrature rule.

\section{Algorithm for extracting minimizers in polynomial optimization problems}

As an application of all the previous results in this section we find a stopping criterion for the moment relaxation hierarchy, in other words, we find a condition on the optimal solution of $(P_{d})$ $L$, such that $L(f)=P^{*}_{d}=P^{*}$. In this this case we also find potencial global minimizers. In \cite{didie}  Henrion and Lasserre  the stopping criterion was $L$ to be flat and in this algorithm the stopping criterium is  $W_{L}^{T}M_{L'}W_{L}$ being Hankel, and  as we have already seen in \ref{otroflat} this condition is more general. It important to point out that despite this condition is more general than being flat we can not ensure optimality until we check that the candidate to minimizers are inside to the basic closed semialgebraic set $S$, condition that it is always possible to ensure if the set $S$ is a set described with linear polynomials, if the set $S$ is $\R^{n}$ or we have flat extension of some degree on the optimal solution, that is to say $\rk M_{d}(y)=\rk M_{s}(y)$ for sufficient small $s$, see  \cite[Theorem 6.18]{lau} or  \cite[Theorem 1.6]{cf2}  for a proof. At the end of this paper we summarize all this results in an algorithm with examples and also we illustrate polynomial optimization problems where this new stopping criterion allow us to conclude optimality even in case where the optimal solution is not flat as we already advance in \ref{porfavor} and in \ref{amigo}.

\begin{teor}\label{popcorn}
Let $f,p_{1},\ldots,p_{m}\in\R[\underline{X}]_{2d}$ and $L$ be an optimal solution of $(P_{2d})$. Suppose that $W_{L}^{T}A_{L}W_{L}$ is a generalized Hankel matrix. Then $L$ has a quadrature rule on $G_{L}$. Moreover, suppose the nodes of the quadrature rule lie on $S$ and $f\in\R[\underline{X}]_{2d-1}$, then $L(f)=P^{*}$ and the nodes  are global minimizers.
\end{teor}
\begin{proof}
Since $W^{T}_{L}A_{L}W_{L}$ is Hankel by Corollary \ref{matrices} and \ref{importante} there exists  exists nodes $a_{1},\ldots,a_{N}\in\R^{n}$ and weights $\lambda_{1}>0,\ldots,\lambda_{N}>0$, where $N:=\dim T_{L}$ such that: 
\begin{equation}
L(p)=\sum_{i=1}^{N}\lambda_{i}p(a_{i})\text{ for all }p\in G_{L}\nonumber
\end{equation}
Moreover if the nodes of this quadrature rule are contained in $S$ by \ref{motiv} (iii) $P^{*}=P^{*}_{d}=f(a_{i})$ for $i\in\{1,\ldots,N\}$.
\end{proof}

The following Lemma was already proved in \cite[lemma 2.7]{rev}. We will use it  to prove  the Corollary \ref{poliedro}.

\begin{lem}\label{interpolation}
Let $L=\sum_{i=1}^{N}\lambda_{i}\ev_{a_{i}}\in\R[\underline{X}]_{2d}^{*}$ for $a_{1},\ldots,a_{N}$ pairwise different points and $\lambda_{1}>0,\ldots,\lambda_{n}>0$ such that $L$ is flat. Then there exist interpolation polynomials $q_{1},\ldots,q_{N}\in\R[\underline{X}]_{2d}$ at the points $a_{1},\ldots,a_{N}$ of degree at most $d-1$.
\end{lem}
\begin{proof}
Let us consider the isometry map \eqref{sigma1} already defined in \ref{fbound}:
\begin{equation}
\sigma_{1}:T_{L} \longrightarrow \frac{\R[\underline{X}]}{U_{\Lambda}},\overline{p}^{L}\mapsto\overline{p}^{\Lambda},\text{ for }p\in\R[\underline{X}]_{d-1} \nonumber
\end{equation}

It is moreover an isomorphism of euclidean vector spaces since $\dim(\frac{\R[\underline{X}]}{U_{\Lambda}})=N$ by Proposition \ref{zero}. It is very well known that there exits interpolation polynomial $h_{1},\ldots,h_{N}\in\R[\underline{X}]$ at the points $a_{1},\ldots,a_{N}$, such that $h_{i}(a_{j})=\delta_{i,j}$ for $i,j\in\{1,\ldots,n\}$. Define $\overline{q_{j}}^{L}:=\sigma^{-1}(\overline{p}^{\Lambda})$ for $q_{j}\in\R[\underline{X}]_{d-1}$. Then for $j\in\{1,\ldots,N\}$:
\begin{align}
0\leq\sum_{i=1}^{N}\lambda_{i}q_{j}^{2}(a_{i})=L(q_{j}^{2})=\left\langle \overline{q_{j}}^{L},\overline{q_{j}}^{L} \right\rangle_{L}=\left\langle \overline{h_{j}}^{\Lambda} , \overline{h_{j}}^{\Lambda} \right\rangle_{\Lambda}=\Lambda(h_{j}^{2})=\lambda_{j}\nonumber
\end{align}
and therefore $q_{j}(a_{i})=\delta_{i,j}$ for $i,j\in\{1,\ldots,N\}$.
\end{proof}

\begin{cor}\label{poliedro}
Let $p_{1},\ldots,p_{m}\in\R[\underline{X}]_{1}$  and $L$ be an optimal solution of $(P_{2d})$ with $f\in\R[\underline{X}]_{2d-1}$. Suppose that $W_{L}^{T}A_{L}W_{L}$ is a Hankel matrix. Then $L$ has a quadrature rule representation on $G_{L}$, $L(f)=P^{*}$ and the nodes are minimizers of $(P)$.
\end{cor}
\begin{proof}
From Theorem \ref{popcorn} there exists  exists nodes $a_{1},\ldots,a_{N}\in\R^{n}$ and weights $\lambda_{1}>0,\ldots,\lambda_{N}>0$, where $N:=\dim T_{L}$ such that: 
\begin{equation}
L(p)=\sum_{i=1}^{N}\lambda_{i}p(a_{i})\text{ for all }p\in G_{L}\nonumber
\end{equation}
To conclude the Corollary by Theorem \ref{popcorn} it is  enough to show that the nodes $a_{1},\ldots,a_{N}$ are contained in $S$. In Theorem \ref{main} we proved that $\hat{L}:=\sum_{i=1}^{N}\lambda_{i}\ev_{a_{i}}\in\R[\underline{X}]^{*}_{2d}$ is flat. Then by the Lemma \ref{interpolation} there are interpolation polynomials $q_{1},\ldots,q_{N}$ at the points $a_{1},\ldots,a_{N}$ having at most degree $d-1$. Since $\deg (q_{i}^{2}p_{j})\leq 2d -1$ then $q_{i}^{2}p_{j}\in T_{2d}(p_{1},\ldots,p_{m})$ and therefore:
\begin{equation}
 0\leq L(q_{i}^{2}p_{j})=\hat{L}(q_{i}^2p_{j})=\lambda_{i}p_{j}(a_{i})\nonumber
\end{equation}
This equality proves that $p_{j}(a_{i})\geq 0$ for  $j\in\{1,\ldots,m \}$ and $i\in\{1,\ldots,N \}$
so we can conclude $\{a_{1},\ldots,a_{N} \}\subseteq S$.
\end{proof}

\begin{rem}
The above results: Theorem \ref{popcorn} and Corollary \ref{poliedro} can be written in terms of an optimal solution of a Moment relaxation of even degree by taking as an optimal solution its restriction to one degree less.
\end{rem}

\begin{algorithm}
  \KwIn{A polynomial optimization problem $(\textbf{P})$ \eqref{Pop}.}
  \KwOut{The minimum $\textbf{P}^{*}$ and minimizers $\textbf{a}_{1},\ldots,\textbf{a}_{r}\subseteq \textbf{S}$ of $\textbf{(P)}$.}
  $\textbf{k}:=\max\{\deg{f},\deg{p_{1}},\ldots,\deg{p_{m}}\}$\;
  Compute an optimal solution $\textbf{M}:=M_{\lfloor\frac{k}{2} \rfloor}(\textbf{y})$ of the moment relaxation $(\textbf{P}_{\textbf{k}})$ and also compute $\textbf{W}_{\textbf{M}}$ matrix such that: $$\textbf{M}=
\left(
\begin{array}{c|c}
\makebox{$A_{M}$}&\makebox{$A_{M}\textbf{W}_{\textbf{M}}$}\\
\hline
  \vphantom{\usebox{0}}\makebox{$\textbf{W}_{\textbf{M}}^{T}A_{M}$}&\makebox{$C_{M}$}
\end{array}
\right)
$$ \;

\If {$\textbf{W}_{\textbf{M}}^{\textbf{T}}\textbf{A}_{\textbf{M}}\textbf{W}_{\textbf{M}}$ is a  Hankel matrix} { go to $\textbf{7}$}
\Else{$\textbf{k:=k+1}$ and go to $\textbf{2}$.}

\If {($\textbf{k}$ even and $f\in\R[\underline{X}]_{k-1})$ \textbf{or} ($\textbf{k}$ odd and $f\in\R[\underline{X}]_{k-2}$) } { go to $\textbf{14}$}
\Else{ \If { $C_{M}=W_{M}A_{M}W_{M}$} {go to $\textbf{14}$} 
        \Else  {$\textbf{k:=k+1}$ go to $\textbf{2}$}    }
Compute the truncated multiplication operators of $\textbf{M}$: $\textbf{A}_{\textbf{1,M}},\ldots,\textbf{A}_{\textbf{n,M}}$ and go to $\textbf{15}$.\;				
Compute an orthonormal basis $\{\textbf{v}_{1},\ldots,\textbf{v}_{r}\}$ of $\textbf{T}_{\textbf{M}}$ of common eigenvectors of the truncated multiplication operators such that $\textbf{A}_{\textbf{i,M}}\textbf{v}_{j}=\textbf{a}_{j,i}\textbf{v}_{j}$ and go to $\textbf{16}$.\;

\If  {$\textbf{a}_{1},\ldots,\textbf{a}_{n}\in\textbf{S}$ }{ go to $\textbf{20}$  }
\Else{ $\textbf{k:=k+1}$ and go to \textbf{2} }
We can conclude that the points $\{\textbf{a}_{1},\ldots,\textbf{ a}_{r} \}\subseteq \textbf{ S}$ are minimizers of $(\textbf{P})$, and $\textbf{P}^{*}=\textbf{f}(\textbf{a}_{i})$ for all $i\in\{1,\ldots,r\}$
  \caption{Algorithm for extracting minimizers of (P)}
  
\end{algorithm}

\section{Software and examples}
To find an optimal solution of the Moment relaxation and for the big calculations we have used the following softwares:
\begin{itemize}
\item YALMIP: developed by J. Löfberg. It is a toolbox for Modeling and Optimization in MATLAB. Published in  the Journal Proceedings of the CACSD Conference in 2004. For more information see: http\slash\slash:yalmip.github.io\slash.
\item SEDUMI: developed by J. F. Sturm. It is a toolbox for optimization over symmetric cones. Published in the Journal Optimization Methods and Software in 1999. For more information see: http:\slash\slash sedumi.ie.lehigh.edu\slash.
\item MPT: the Multi-parametric Toolbox is an open source, Matlab-based toolbox for parametric optimization, computational geometry and model predictive control. For more information see: http:\slash\slash people.ee.ethz.ch\slash~mpt\slash3\slash.
\item MATLAB and Statistics Toolbox Release 2016a, The MathWorks, Inc., Natick, Massachusetts, United States.
\end{itemize}

\begin{ej}\label{Rosenbrock}
Let us apply the algorithm to the following polynomial optimization problem, taken from \cite{tonto} :
\begin{equation*}
\begin{aligned}
& {\text{minimize}}
& & f(x)=100(x_{2}-x_{1}^2)^2+100(x_{3}-x_{2}^2)^2+(x_{1}-1)^2+(x_{2}-1)^2 \\
& \text{subject to}
& & -2.048\leq x_{1}\leq 2.048\\
& 
& & -2.048\leq x_{2}\leq 2.048\\
&
& & -2.048\leq x_{3}\leq 2.048 \\
\end{aligned}
\end{equation*}
We initialize $\textbf{k}=4$ and compute an optimal solution of the moment relaxation $(P_{4})$. In this case reads as:
{\scriptsize	
\begin{equation}
\textbf{M}:=M_{4,1}(y) =
\begin{blockarray}{ccccccccccc}
\text{ } & 1 & X_{1} & X_{2} & X_{3} & X_{1}^2 & X_{1}X_{2} & X_{1}X_{3} & X_{2}^2 & X_{2}X_{3} & X_{3}^2\\
\begin{block}{c(cccc|cccccc)}
1 &   1.0000  &  1.0000  &  1.0000  &  1.0000   & 1.0000   & 1.0000  &  1.0000 &  1.0000  &  1.0000 &   1.0000 \\
 X_{1} &   1.0000  &  1.0000  &  1.0000  &  1.0000   & 1.0000   & 1.0000  &  1.0000 &  1.0000  & 1.0000  &  1.0000 \\
  X_{2} &  1.0000 &   1.0000  &  1.0000  &  1.0000   & 1.0000   &  1.0000 &   1.0000 & 1.0000 &   1.0000 &   1.0000 \\
  X_{3} & 1.0000 &   1.0000  &  1.0000  &  1.0000   &  1.0000  &  1.0000 &   1.0000 &  1.0000 &   1.0000 &    1.0000 \\ \cline{2-11}
 X_{1}^2& 1.0000 &   1.0000  &  1.0000  &  1.0000   &  1.0000  &  1.0000 &   1.0000 &  1.0000 &   1.0000 &    1.0000 \\
 X_{1}X_{2} & 1.0000 &   1.0000  &  1.0000  &  1.0000   &  1.0000  &  1.0000 &   1.0000 &  1.0000 &   1.0000 &    1.0000 \\
X_{1}X_{3} & 1.0000 &   1.0000  &  1.0000  &  1.0000   &  1.0000  &  1.0000 &   1.0000 &  1.0000 &   1.0000 &    1.0000 \\
X_{2}^2  & 1.0000 &   1.0000  &  1.0000  &  1.0000   &  1.0000  &  1.0000 &   1.0000 &  1.0000 &   1.0000 &    1.0000 \\
X_{2}X_{3} &  1.0000 &   1.0000  &  1.0000  &  1.0000   &  1.0000  &  1.0000 &   1.0000 &  1.0000 &   1.0000 &    1.0000 \\
X_{3}^{2}  & 1.0000 &   1.0000  &  1.0000  &  1.0000   &  1.0000  &  1.0000 &   1.0000 &  1.0000 &   1.0000 &   \textbf{5.6502} \\
 \end{block}
    \end{blockarray}\nonumber
	\end{equation}}
We can calculate that:	
{\scriptsize
\begin{equation}
W^{T}_{\textbf{M}}A_{\textbf{M}}W_{\textbf{M}} =\left( \begin{array}{cccccc}
1.0000 &   1.0000  &  1.0000  &  1.0000   &  1.0000  &  1.0000 \\
 1.0000 &   1.0000  &  1.0000  &  1.0000   &  1.0000  &  1.0000 \\
 1.0000 &   1.0000  &  1.0000  &  1.0000   &  1.0000  &  1.0000 \\
 1.0000 &   1.0000  &  1.0000  &  1.0000   &  1.0000  &  1.0000 \\
 1.0000 &   1.0000  &  1.0000  &  1.0000   &  1.0000  &  1.0000 \\
 1.0000 &   1.0000  &  1.0000  &  1.0000   &  1.0000  & \textbf {1.0000}
\end{array}
\right) \nonumber
\end{equation}}
is a Hankel matrix but $f\notin\R[X_{1},X_{2},X_{3}]_{3}$ and $C_{\textbf{M}}\neq W^{T}_{\textbf{M}}A_{\textbf{M}}W_{\textbf{M}}$, so we need to try again with $\textbf{k}=5$ and in this case the solution of the moment relaxation $(P_{5})$ reads as:

{\scriptsize	
\begin{equation}
\textbf{M}:=M_{5,1}(y) =
\begin{blockarray}{ccccccccccc}
\text{ } & 1 & X_{1} & X_{2} & X_{3} & X_{1}^2 & X_{1}X_{2} & X_{1}X_{3} & X_{2}^2 & X_{2}X_{3} & X_{3}^2\\
\begin{block}{c(cccc|cccccc)}
1 &   1.0000  &  1.0000  &  1.0000  &  1.0000   & 1.0000   & 1.0000  &  1.0000 &  1.0000  &  1.0000 &   1.0000 \\
 X_{1} &   1.0000  &  1.0000  &  1.0000  &  1.0000   & 1.0000   & 1.0000  &  1.0000 &  1.0000  & 1.0000  &  1.0000 \\
  X_{2} &  1.0000 &   1.0000  &  1.0000  &  1.0000   & 1.0000   &  1.0000 &   1.0000 & 1.0000 &   1.0000 &   1.0000 \\
  X_{3} & 1.0000 &   1.0000  &  1.0000  &  1.0000   &  1.0000  &  1.0000 &   1.0000 &  1.0000 &   1.0000 &    1.0000 \\ \cline{2-11}
 X_{1}^2& 1.0000 &   1.0000  &  1.0000  &  1.0000   &  1.0000  &  1.0000 &   1.0000 &  1.0000 &   1.0000 &    1.0000 \\
 X_{1}X_{2} & 1.0000 &   1.0000  &  1.0000  &  1.0000   &  1.0000  &  1.0000 &   1.0000 &  1.0000 &   1.0000 &    1.0000 \\
X_{1}X_{3} & 1.0000 &   1.0000  &  1.0000  &  1.0000   &  1.0000  &  1.0000 &   1.0000 &  1.0000 &   1.0000 &    1.0000 \\
X_{2}^2  & 1.0000 &   1.0000  &  1.0000  &  1.0000   &  1.0000  &  1.0000 &   1.0000 &  1.0000 &   1.0000 &    1.0000 \\
X_{2}X_{3} &  1.0000 &   1.0000  &  1.0000  &  1.0000   &  1.0000  &  1.0000 &   1.0000 &  1.0000 &   1.0000 &    1.0000 \\
X_{3}^{2}  & 1.0000 &   1.0000  &  1.0000  &  1.0000   &  1.0000  &  1.0000 &   1.0000 &  1.0000 &   1.0000 &   \textbf{1.0014} \\
 \end{block}
    \end{blockarray}\nonumber
	\end{equation}}
	
we can calculate that:
{\scriptsize
\begin{equation}
W^{T}_{\textbf{M}}A_{\textbf{M}}W_{\textbf{M}} =\left( \begin{array}{cccccc}
1.0000 &   1.0000  &  1.0000  &  1.0000   &  1.0000  &  1.0000 \\
 1.0000 &   1.0000  &  1.0000  &  1.0000   &  1.0000  &  1.0000 \\
 1.0000 &   1.0000  &  1.0000  &  1.0000   &  1.0000  &  1.0000 \\
 1.0000 &   1.0000  &  1.0000  &  1.0000   &  1.0000  &  1.0000 \\
 1.0000 &   1.0000  &  1.0000  &  1.0000   &  1.0000  &  1.0000 \\
 1.0000 &   1.0000  &  1.0000  &  1.0000   &  1.0000  &  \textbf{1.0000}
\end{array}
\right) \nonumber
\end{equation}}

is a Hankel matrix but $f\notin\R[X_{1},X_{2},X_{3}]_{3}$ and $C_{\textbf{M}}\neq W^{T}_{\textbf{M}}A_{\textbf{M}}W_{\textbf{M}}$. In this case if we rounding we can consider $C_{\textbf{M}}=W^{T}_{\textbf{M}}A_{\textbf{M}}W_{\textbf{M}}$, i.e. $M$ flat, and continue with the algorithm and we could obtain already the minimizers, but to be more precise let us increase to $\textbf{k}=6$ and we get the following  optimal solution in the moment relaxation $(P_{6})$:
\begin{equation}\label{rosenbrock}
 \textbf{M}:=M_{6,1}(y)=
\left(
\begin{array}{c|c}
\makebox{$A_{\textbf{M}}$}&\makebox{$A_{\textbf{M}}W_{\textbf{M}}$}\\
\hline
  \vphantom{\usebox{0}}\makebox{$W_{\textbf{M}}^{T}A_{\textbf{M}}$}&\makebox{$C_{\textbf{M}}$}
\end{array}
\right)
\end{equation}
where:

{\scriptsize	
\begin{equation}
A_{\textbf{M}} = \begin{blockarray}{ccccccccccc}
\text{ } & 1 & X_{1} & X_{2} & X_{3} & X_{1}^2 & X_{1}X_{2} & X_{1}X_{3} & X_{2}^2 & X_{2}X_{3} & X_{3}^2\\
\begin{block}{c(cccccccccc)}
1 &    1.0000  &  1.0000  &  1.0000  &  1.0000   & 1.0000   & 1.0000  &  1.0000 &  1.0000  &  1.0000 &   1.0000 \\
X_{1} &    1.0000  &  1.0000  &  1.0000  &  1.0000   & 1.0000   & 1.0000  &  1.0000 &  1.0000  & 1.0000  &  1.0000 \\
X_{2} &    1.0000 &   1.0000  &  1.0000  &  1.0000   & 1.0000   &  1.0000 &   1.0000 & 1.0000 &   1.0000 &   1.0000 \\
X_{3} &    1.0000 &   1.0000  &  1.0000  &  1.0000   &  1.0000  &  1.0000 &   1.0000 &  1.0000 &   1.0000 &    1.0000 \\
X_{1}^2 &  1.0000 &   1.0000  &  1.0000  &  1.0000   &  1.0000  &  1.0000 &   1.0000 &  1.0000 &   1.0000 &    1.0000 \\
X_{1}X_{2} &  1.0000 &   1.0000  &  1.0000  &  1.0000   &  1.0000  &  1.0000 &   1.0000 &  1.0000 &   1.0000 &    1.0000 \\
X_{1}X_{3} &  1.0000 &   1.0000  &  1.0000  &  1.0000   &  1.0000  &  1.0000 &   1.0000 &  1.0000 &   1.0000 &    1.0000 \\
X_{2}^2 &  1.0000 &   1.0000  &  1.0000  &  1.0000   &  1.0000  &  1.0000 &   1.0000 &  1.0000 &   1.0000 &    1.0000 \\
X_{2}X_{3} &  1.0000 &   1.0000  &  1.0000  &  1.0000   &  1.0000  &  1.0000 &   1.0000 &  1.0000 &   1.0000 &    1.0000 \\
X_{3}^2 &  1.0000 &   1.0000  &  1.0000  &  1.0000   &  1.0000  &  1.0000 &   1.0000 &  1.0000 &   1.0000 &    1.0000\\
 \end{block}
\end{blockarray}
  \nonumber
	\end{equation}}

{\scriptsize	
\begin{equation}
W_{\textbf{M}} = \left(
\begin{array}{cccccccccc}
1 & 0 & 0 & 0 & 0 & 0 & 0 & 0 & 0 & 0 \\
0 & 1 & 0 & 0 & 0 & 0 & 0 & 0 & 0 & 0 \\
0 & 0 & 1 & 0 & 0 & 0 & 0 & 0 & 0 & 0 \\
0 & 0 & 0 & 1 & 0 & 0 & 0 & 0 & 0 & 0  \\
0 & 0 & 0 & 0 & 1 & 0 & 0 & 0 & 0 & 0  \\
0 & 0 & 0 & 0 & 0 & 1 & 0 & 0 & 0 & 0  \\
0 & 0 & 0 & 0 & 0 & 0 & 1 & 0 & 0 & 0  \\
0 & 0 & 0 & 0 & 0 & 0 & 0 & 1 & 0 & 0  \\
0 & 0 & 0 & 0 & 0 & 0 & 0 & 0 & 1 & 0  \\
0 & 0 & 0 & 0 & 0 & 0 & 0 & 0 & 0 & 1
 \end{array}
    \right)\nonumber
	\end{equation}}
 and 

{\scriptsize	
\begin{equation}
c_{\textbf{M}} =
\begin{blockarray}{ccccccccccc}
\text{ } & X_{1}^3 & X_{1}^2X_{2} & X_{1}^{2}X_{3} & X_{1}X_{2}^{2} & X_{1}X_{2}X_{3} &X_{1}X_{3}^{2} & X_{2}^{3} & X_{2}^{2}X_{3} & X_{2}X_{3}^{2} & X_{3}^{2} \\  
\begin{block}{c(cccccccccc)}
 X_{1}^{3} &   5.2880 &   0.9994  &  0.9994  &  2.4826  &  0.9989  &  2.4744  &  0.9988  &  1.0004  &  1.0020  &  1.0001 \\
X_{1}^{2}X_{2} &    0.9994 &   2.4826  &  0.9989  &  0.9988  &  1.0004  &  1.0020  &  2.4832  &  1.0010  &  1.6671  &  1.0007 \\
X_{1}^{2}X_{3} &    0.9994 &   0.9989  &  2.4744  &  1.0004  &  1.0020  &  1.0001  &  1.0010  &  1.6671  &  1.0007  &  2.4638 \\
X_{1}X_{2}^{2} &    2.4826 &   0.9988  &  1.0004  &  2.4832  &  1.0010  &  1.6671  &  0.9983  &  1.0007  &  1.0015  &  1.0001 \\
X_{1}X_{2}X_{3} &    0.9989 &   1.0004  &  1.0020  &  1.0010  &  1.6671  &  1.0007  &  1.0007  &  1.0015  &  1.0001  &  1.0016 \\
 X_{1}X_{3}^{2} &   2.4744 &   1.0020  &  1.0001  &  1.6671  &  1.0007  &   2.4638 &   1.0015 &   1.0001 &   1.0016 &   0.9912 \\
X_{2}^{3} &    0.9988 &   2.4832  &  1.0010  &  0.9983  &  1.0007  &  1.0015  &  5.2883  &  1.0071  &  2.4669  &  1.0072 \\
 X_{2}^{2}X_{3} &   1.0004 &   1.0010  &  1.6671  &  1.0007  &  1.0015  &  1.0001  &  1.0071  &  2.4669  &  1.0072  &  2.4579 \\
X_{2}X_{3}^{2}  &  1.0020 &  1.6671   &  1.0007  &  1.0015  &  1.0001  &  1.0016  &  2.4669  &  1.0072  &  2.4579  &  1.0040 \\
X_{3}^{3} &    1.0001 &   1.0007 &   2.4638  &  1.0001  &  1.0016  &  0.9912  &  1.0072 &   2.4579  &  1.0040  & 14.6604\\
\end{block}
\end{blockarray}
\nonumber
	\end{equation}}
	
We calculate that:

{\scriptsize	
\begin{equation}
W^{T}_{\textbf{M}}A_{\textbf{M}}C_{\textbf{M}}= \left(
\begin{array}{cccccccccc}

    1.0000  &  1.0000  &  1.0000  &  1.0000   & 1.0000   & 1.0000  &  1.0000 &  1.0000  &  1.0000 &   1.0000 \\
    1.0000  &  1.0000  &  1.0000  &  1.0000   & 1.0000   & 1.0000  &  1.0000 &  1.0000  & 1.0000  &  1.0000 \\
    1.0000 &   1.0000  &  1.0000  &  1.0000   & 1.0000   &  1.0000 &   1.0000 & 1.0000 &   1.0000 &   1.0000 \\
    1.0000 &   1.0000  &  1.0000  &  1.0000   &  1.0000  &  1.0000 &   1.0000 &  1.0000 &   1.0000 &    1.0000 \\
  1.0000 &   1.0000  &  1.0000  &  1.0000   &  1.0000  &  1.0000 &   1.0000 &  1.0000 &   1.0000 &    1.0000 \\
  1.0000 &   1.0000  &  1.0000  &  1.0000   &  1.0000  &  1.0000 &   1.0000 &  1.0000 &   1.0000 &    1.0000 \\
  1.0000 &   1.0000  &  1.0000  &  1.0000   &  1.0000  &  1.0000 &   1.0000 &  1.0000 &   1.0000 &    1.0000 \\
  1.0000 &   1.0000  &  1.0000  &  1.0000   &  1.0000  &  1.0000 &   1.0000 &  1.0000 &   1.0000 &    1.0000 \\
  1.0000 &   1.0000  &  1.0000  &  1.0000   &  1.0000  &  1.0000 &   1.0000 &  1.0000 &   1.0000 &    1.0000 \\
  1.0000 &   1.0000  &  1.0000  &  1.0000   &  1.0000  &  1.0000 &   1.0000 &  1.0000 &   1.0000 &    1.0000\\
 \end{array}
    \right)\nonumber
	\end{equation}}
is a generalized Hankel matrix and $f\in\R[X_{1},X_{2},X_{3}]_{5}$. By Theorem \ref{poliedro} we have optimality with optimal value $P^{*}=P^{*}_{6}=2.3527\cdot 10^{-8}\approx 0$. Finally we get that the matrices of the truncated GNS operators with respect to the orthonormal basis $v:=\left\langle 1 \right\rangle_{\textbf{M}}$ are:
\begin{equation}
M(M_{\textbf{M},X_{1}},v)=\left( 1 \right), M(M_{\textbf{M},X_{2}},v)=\left( 1 \right)\text{ and } M(M_{\textbf{M},X_{2}},v)=\left( 1 \right)\nonumber
\end{equation}

The operators are in diagonal form so we have already an orthonormal basis of $T_{\textbf{M}}$ of common eigenvectors of the truncated GNS operators of $\textbf{M}$ $v:=\left\langle 1 \right\rangle_{\textbf{M}}$, then a global minimizer is $\left(1,1,1\right)\in\R^{n}$, and:
\begin{equation}
\widetilde{\textbf{M}}=V_{3}(1,1,1)V_{3}^{T}(1,1,1).\nonumber
\end{equation}
\end{ej}

\begin{ej}\label{nonconvex}
Let us consider the following polynomial optimization problem, defined on a non convex closed semialgebraic set, taken from \cite[problem 4.6]{griegos} : 
\begin{equation*}
\begin{aligned}
& {\text{minimize}}
& & f(x)=-x_{1}-x_{2}\\
& \text{subject to}
& & x_{2}\leq 2x_{1}^4-8x_{1}^3+8x_{1}^2+2\\
& 
& & x_{2}\leq 4x_{1}^4-32x_{1}^3+88x_{1}^2-96x_{1}+36\\
&
& & 0\leq x_{1} \leq 3\\
&
& & 0\leq x_{2} \leq 4\\
\end{aligned}
\end{equation*}
We initialize $\textbf{k}=4$. An optimal solution of $(P_{4})$ reads as:

\begin{equation}\label{cuartoej}
\textbf{M}:=M_{4,1}(y) = \begin{blockarray}{ccccccc}
\text{ } & 1 & X_{1} & X_{2} & X_{1}^2 & X_{1}X_{2} & X_{2}^{2}\\
\begin{block}{c(ccc|ccc)}
1 & 1.0000  &  3.0000 &   4.0000 &   9.0000 &  12.0000 &  16.0000 \\
X_{1} & 3.0000  &  9.0000  & 12.0000 &  27.0000 &  36.0000 &  48.0000 \\
X_{2} & 4.0000 &  12.0000 &  16.0000 &  36.0000 &  48.0000 &  64.0000\\ \cline{2-7}
X_{1}^{2} &  9.0000  & 27.0000 &   36.0000 &  107.6075 &  109.0814 &  176.3211\\
X_{1}X_{2} &   12.0000 &  36.0000 &  48.0000 &  109.0814 & 176.3211 &  194.9661\\
X_{2}^{2} &   16.0000 &  48.0000 &  64.0000 &  176.3211 &  194.9661 &  368.5439\\
 \end{block}
\end{blockarray}
	\end{equation}

and 

\begin{equation}
\widetilde{\textbf{M}} = \begin{blockarray}{ccccccc}
&1 & X_{1} & X_{2} & X_{1}^{2} & X_{1}X_{2} & X_{2}^{2} \\
\begin{block}{c(ccc|ccc)}
1 &1.0000  &  3.0000 &   4.0000 &   9.0000 &  12.0000 &  16.0000 \\
X_{1} & 3.0000  &  9.0000  & 12.0000 &  27.0000 &  36.0000 &  48.0000 \\
X_{2} & 4.0000 &  12.0000 &  16.0000 &  36.0000 &  48.0000 &  64.0000\\ \cline{2-7}
X_{1}^{2} &    9.0000  & 27.0000 &   36.0000 & 81.000 &108.000 & 144.000 \\
X_{1}X_{2} &   12.0000 &  36.0000 &  48.0000 & 108.000 & 144.000 &192.00 \\
X_{2}^{2} &   16.0000 &  48.0000 &  64.0000 & 144.000 & 192.000 & 256.000\\
\end{block}
\end{blockarray}
	\end{equation}

taking for example:

\begin{center}
$W_{\textbf{M}} = \left(
\begin{array}{lll}

   9 & 0 & 0\\
	 0 & 0 & 0 \\
   0 & 3 & 4 \\

   \end{array}
    \right)$
\end{center}

 $\widetilde{\textbf{M}}$ is a generalized Hankel matrix. The matrix of the truncated GNS multiplaction operators with respect to the orthonormal basis $v=\left\langle \overline{1}^{\textbf{M}} \right\rangle$ are:
\begin{center}
$M(M_{\textbf{M},1},v) = \left(
\begin{array}{l}
 3 \\
   \end{array}
    \right)$ and $M(M_{\textbf{M},1},v)= \left(
\begin{array}{l}
 4 \\
   \end{array}
    \right)$
\end{center}

Hence the candidate to minimizer is $\left(3,4\right)$, however it does not lie in $S$, then $\left(3,4\right)$ cannot be a minimizer and $f(3,4)=-7$ cannot be the minimum. Then we try with a relaxation of order $\textbf{k}=5$. An optimal solution of the moment relaxation $(P_{5})$ is the following:

\begin{equation}\label{otroejemplo}
\textbf{M}:=M_{5,1}(y)= \begin{blockarray}{ccccccc}
\text{ } & 1 & X_{1} & X_{2} & X_{1}^2 & X_{1}X_{2} & X_{2}^{2}\\
\begin{block}{c(ccc|ccc)}
  1 &       1.00      &    2.67   &       4.00   &       8.00   &      10.67     &    16.00 \\
  X_{1} &        2.67     &     8.00  &       10.67  &       24.00  &       32.00    &     42.67 \\
 X_{2} &         4.00     &    10.67  &       16.00  &       32.00  &       42.67    &     64.00\\ \cline{2-7}
X_{1}^{2} &          8.00     &    24.00  &       32.00  &       72.00  &       96.00    &    128.00 \\
  X_{1}X_{2} &       10.67    &     32.00  &       42.67  &       96.00  &      128.00    &    170.67 \\
  X_{2}^{2} &       16.00    &     42.67  &       64.00  &      128.00  &      170.67    &    256.00 \\
	\end{block}
	\end{blockarray}
				\end{equation}

In this case $C_{\textbf{M}}=W_{\textbf{M}}^{T}A_{\textbf{M}}W_{\textbf{M}}$, therefore $\textbf{M}$ is flat and in particular the operators commute by \ref{importante2}.  After the simultaneous diagonalization of the truncated GNS operators we get that the candidate to minimizers are $(0,4)\notin S$ and $(3,4)\notin S$. Hence we try with a relaxation of order $\textbf{k}=6$. An optimal solution of the moment relaxation $(P_{6})$ reads as:
{\scriptsize
\begin{equation}\label{nopuedomas}
\textbf{M}:=M_{6,1}(y)=
\begin{blockarray}{ccccccccccc}
\text{ } & 1 & X_{1} & X_{2} & X_{1}^{2} & X_{1}X_{2} & X_{2}^{2} & X_{1}^{3} & X_{1}^{2}X_{2} & X_{1}X_{2}^{2} & X_{2}^{3} \\
\begin{block}{c(cccccc|cccc)}
 1 & 1.00      &    2.67     &     4.00    &      8.00   &      10.67  &       16.00   &      24.00   &      32.00 &          42.67     &    64.00 \\
X_{1} &          2.67   &  8.00    &     10.67    &     24.00   &      32.00   &      42.67  &       72.00  &       96.00   &     128.00   &     170.67 \\
 X_{2} &         4.00  &        10.67  &       16.00    &     32.00   &      42.67   &      64.00    &     96.00    &    128.00  &      170.67     &   256.00 \\
 X_{1}^{2} &         8.00    &     24.00   &      32.00   &      72.00    &     96.00    &    128.00   &     216.00 &       288.00  &      384.00  &      512.00 \\
  X_{1}X_{2} &       10.67    &     32.00   &      42.67   &      96.00   &     128.00   &     170.67    &    288.00 &       384.00     &   512.00  &      682.66 \\
  X_{2}^{2} &       16.00    &     42.67   &      64.00   &     128.00   &     170.67    &    256.00  &      384.00 &       512.00  &      682.66  &     1024.00 \\ \cline{2-11}
   X_{1}^{3} &      24.00  &       72.00   &      96.00    &    216.00   &     288.00    &    384.00  &   204299.70 &       870.25  &    19035.69  &     1583.15 \\
   X_{1}^{2}X_{2} &      32.00   &      96.00  &      128.00   &     288.00  &      384.00   &     512.00  &      870.25 &     19035.69  &     1583.15 &     18023.54 \\
 X_{1}X_{2}^{2}  &      42.67   &     128.00  &      170.67   &     384.00  &      512.00  &      682.66   &   19035.69 &      1583.15   &   18023.54  &     2822.34 \\
X_{2}^{3} &         64.00  &      170.67  &      256.00   &     512.00   &     682.66  &     1024.00   &    1583.15  &    18023.54  &     2822.34 &     58336.42 \\
\end{block}
\end{blockarray}
\end{equation}}

and
\begin{equation}
W_{\textbf{M}}^{T}A_{\textbf{M}}W_{\textbf{M}}=\begin{blockarray}{ccccc}
& X_{1}^{3} & X_{1}^{2}X_{2} & X_{1}X_{2}^{2} & X_{2}^{3}\\
\begin{block}{c(cccc)}
   X_{1}^{3} &     648.00   &      863.99  &     1151.99  &     1535.99 \\
 X_{1}^{2}X_{2} &       863.99   &    1151.99   &    1535.99   &    2047.99 \\
 X_{1}X_{2}^{2} &      1151.99   &    1535.99   &    2047.99   &    2730.65 \\
 X_{2}^{3} &      1535.99   &    2047.99   &    2730.65   &    4095.99 \\
			\end{block}
			\end{blockarray}\nonumber
			\end{equation}
			
is a Hankel matrix. However we get the same candidate to minimizers as in the previous relaxation which does not belong to $S$. Finally we increase to $\textbf{k}=7$, and we get after rounding, the following optimal solution of $(P_{7})$:

{\scriptsize
\begin{equation}
\textbf{M}:=M_{7,1}(y)=\begin{blockarray}{ccccccccccc}
\text{ } & 1 & X_{1} & X_{2} & X_{1}^{2} & X_{1}X_{2} & X_{2}^{2} & X_{1}^{3} & X_{1}^{2}X_{2} & X_{1}X_{2}^{2} & X_{2}^{3} \\
\begin{block}{c(cccccc|cccc)}
1 &   1.00      &    2.33    &      3.18     &     5.43   &       7.40   &      10.10  &       12.64  &       17.25 &        23.53    &     32.11 \\
 X_{1} &        2.33    &      5.43   &       7.40     &    12.64  &       17.25     &    23.53   &      29.45   &      40.18   &      54.82    &     74.80 \\
 X_{2} &         3.18    &      7.40   &      10.10      &   17.25   &      23.53   &      32.11   &      40.18     &    54.82   &      74.80   &     102.07 \\
  X_{1}^{2} &        5.43    &     12.64   &      17.25   &      29.45   &      40.18   &      54.82   &      68.60   &      93.60   &     127.72   &     174.26 \\
  X_{1}X_{2} &        7.40   &      17.25    &     23.53    &     40.18   &      54.82   &      74.80   &      93.60  &      127.72     &   174.26    &    237.77 \\
  X_{2}^{2} &       10.10  &       23.53   &      32.11  &       54.82  &       74.80   &     102.07   &     127.72 &       174.26    &    237.77  &      324.42 \\ \cline{2-11}
  X_{1}^{3} &       12.64   &      29.45    &     40.18   &      68.60  &       93.60   &     127.72  &      159.81 &       218.05   &     297.51    &    405.94 \\
   X_{1}^{2}X_{2} &      17.25    &     40.18    &     54.82    &     93.60  &      127.72  &      174.26   &     218.05  &      297.51    &    405.94     &   553.88 \\
  X_{1}X_{2}^{2} &       23.53    &     54.82     &    74.80    &    127.72  &      174.26    &    237.77   &     297.51  &      405.94   &     553.88  &      755.74 \\
  X_{2}^{3} &       32.11    &     74.80   &     102.07   &     174.26  &      237.77  &      324.42   &     405.94  &      553.88   &     755.74  &     1031.16\\
				\end{block}
				\end{blockarray}
 \nonumber
\end{equation}}

It holds that $\widetilde{\textbf{M}}=\textbf{M}$, therefore in particular $\widetilde{\textbf{M}}$ is a generalized Hankel matrix and the truncated multiplication operators commute. The matrices of the truncated GNS multiplication operators with respect to the orthonormal basis $v:=\{\overline{1}^{\textbf{M}} \}$ are:

\begin{center}
$M(M_{\textbf{M},X_{1}},v) = \left(
\begin{array}{l}
2.3295

   \end{array}
    \right)$
		 and 
	$M(M_{\textbf{M},X_{2}},v)  = \left(
\begin{array}{l}
3.1785

   \end{array}
    \right)$
			
\end{center}

Since $\left( 2.3295, 3.1785 \right)\in S$ then it is also a minimizer and we proved optimality $P^{*}=P^{*}_{7}=-5.5080$.
\end{ej}

\begin{ej}\label{flatcase}
Let us considerer the following polynomial optimization problem taken from \cite[example 5]{las}:
\begin{equation*}
\begin{aligned}
& \text{minimize}
& & f(x)=-(x_{1}-1)^2-(x_{1}-x_{2})^2-(x_{2}-3)^2 \\
& \text{subject to}
& & 1-(x_{1}-1)^2\geq 0\\
& 
& & 1-(x_{1}-x_{2})^2\geq 0\\
&
& & 1-(x_{2}-3)^2\geq 0 \\
\end{aligned}
\end{equation*}

For $k=2$ and $k=3$ in the algorithm, the modified moment matrix of the optimal solution of the Moment relaxation is generalized Hankel and we get as a potencial minimizers, after the truncated GNS construction, $\left(1.56,2.18\right)\in S$ in both relaxations, however $f\notin\R[X_{1},X_{2}]_{1}$ so we can not conclude $\left(1.56,2.18\right)$ is a global minimum. When we increase to $k=4$, and compute an optimal solution  of the moment relaxation $(P_{4})$. We get:

\begin{equation}
\textbf{M}:=M_{4}(y) = 
\begin{blockarray}{ccccccc}
& 1 & X_{1} & X_{2} & X_{1}^2 & X_{1}X_{2} & X_{2}^{2}\\
\begin{block}{c(ccc|ccc)}
1 &   1.0000  &   1.4241  &  2.1137  &   2.2723 &   3.0755 &   4.5683 \\
X_{1} &    1.4241  &   2.2723  &  3.0755  &   3.9688 &   4.9993 &   6.8330 \\
X_{2} &   2.1137  &   3.0755  &  4.5683  &  4.9993  &   6.8330 &  10.1595 \\ \cline{2-7}
X_{1}^{2} &    2.2723  &   3.9688  &  4.9993  &  7.3617  &   8.8468 &   11.3625 \\
X_{1}X_{2} &    3.0755  &   4.9993  &  6.8330  &  8.8468  &   11.3625 &   15.7120 \\
X_{2}^{2} &   4.5683  &   6.8330  & 10.1595  &  11.3625 &  15.7120 &  23.3879 \\
\end{block}
\end{blockarray} \nonumber
	\end{equation}

and we can verify $\widetilde{\textbf{M}}=\textbf{M}$. Hence in this case $\textbf{M}$ is flat, then it is clear that $\widetilde{\textbf{M}}$ is a generalized Hankel matrix implying that the truncated GNS multiplication operators of $\textbf{M}$ commute. We proceed to do the truncated GNS construction and we get the following orthonormal basis of $W_{\textbf{M}}$:
\begin{equation}
W_{\textbf{M}}=\left\langle \overline{1}^{\textbf{M}},\overline{-2.08816+2.0234X_{1}}^{\textbf{M}},\overline{-6.0047-0.9291X_{1}+3.4669X_{2}}^{\textbf{M}} \right\rangle \nonumber
\end{equation}
Denote $v:=\{ \overline{1}^{\textbf{M}},\overline{-2.08816+2.0234X_{1}}^{\textbf{M}},\overline{-6.0047-0.9291X_{1}+3.4669X_{2}}^{\textbf{M}}\}$ such a basis. Then the transformation matrices of the truncated GNS multiplication operators with respect to this basis are:

\begin{equation}
A_{1}:=M(M_{\textbf{M},X_{1}},v)=\left(
\begin{array}{ccc}

    1.4241  &  0.4942 &    0.0000 \\
    0.4942  &  1.5759 &   0.0000 \\
    0.0000  &  0.0000 &   2.0000 \\
\end{array}
\right)\nonumber
\end{equation}

\begin{equation}
A_{2}:=M(M_{\textbf{M},X_{2}},v)=\left(
\begin{array}{ccc}
    2.1137  &  0.1324  &  0.2884 \\
    0.1324  &  2.1543  &  0.3361 \\
    0.2884  &  0.3361  &  2.7320 \\
\end{array}
\right)\nonumber
\end{equation}

Again we follow the same idea as in \cite[algorithm 4.1 Step 1]{japoneses} to apply simultaneous diagonalization to the matrices $A_{1}$ and $A_{2}$. For this we find the orthogonal matrix $P$ that diagonalize a matrix of the following form: 
\begin{equation}
A=r_{1}A_{1}+r_{2}A_{2}\text{ where }r_{1}^{2}+r_{2}^{2}=1\nonumber
\end{equation}
For 
\begin{equation}
P=\left(
\begin{array}{rrr}
    0.7589  &  0.5572  &  0.3371 \\
   -0.6512  &  0.6493  &  0.3929 \\
    0.0000  & -0.5177  &  0.8556 
\end{array}
\right)\nonumber
\end{equation}

we get the following diagonal matrices:
\begin{equation}
P^{T}A_{1}P=\left(
\begin{array}{rrr}
    1.0000  &  0.0000 &  -0.0000 \\
    0.0000  &   2.0000 &   0.0000 \\
   -0.0000  &  0.0000 &   2.0000 
\end{array}
\right),
P^{T}A_{2}P=\left(
\begin{array}{rrr}
 2.0000 &  -0.0000  &  0.0000 \\
   -0.0000 &    2.0000 &  -0.0000 \\
    0.0000 &  -0.0000  &  3.0000 
\end{array}
\right)\nonumber
\end{equation}
 
and with the operation: 

\begin{equation}
P^{T}\left(\begin{array}{c}
   1  \\ 
   0 \\
   0  \\ 
    \end{array}
    \right)=\left(\begin{array}{c}
   0.7589\\ 
   0.5572   \\
   0.3371  \\ 
    \end{array}
    \right)\nonumber
\end{equation}
 we get the square roots of the weights of the quadrature formula. Then we have the following decomposition:
\begin{equation}
\textbf{M}=\widetilde{\textbf{M}}=0.5759V_{2}(1,2)V_{2}^{T}(1,2)+0.3105V_{2}(2,2)V_{2}^{T}(2,2)+0.1137V_{2}(2,3)V_{2}^{T}(2,3)\nonumber
\end{equation}

In this case the points $\left(1,2\right)$,$\left(2,2\right)$, and $\left(2,3 \right)$ lie on $S$, as we already know since it holds the condition of the Theorem 1.6 in \cite{cf2}, and therefore they are global minimizers of $(P)$, and the minimum is $P^{*}=P^{*}_{4}=-2$.

\end{ej}

\section*{Acknowledgements}
This is part of the ongoing research for the authors PhD thesis. I want to thank Markus Schweighofer for supervising my thesis and many fruitful discussions.

\end{document}